\documentclass[11pt,a4paper]{article}
\usepackage[english]{babel}
\usepackage{amsmath,amssymb}
\usepackage{bm}
\usepackage{graphics}
\usepackage{tikz}
\usepackage{url}
\usepackage[ruled]{algorithm2e}
\usepackage{float}

\newtheorem{Lem}{Lemma}[section]
\newtheorem{Prop}[Lem]{Proposition}
\newtheorem{Cor}[Lem]{Corollary}
\newtheorem{Thm}[Lem]{Theorem}

\newtheorem{Discus}[Lem]{Discussion}
\newtheorem{Def}[Lem]{Definition}
\newtheorem{Rem}[Lem]{Remark}

\newtheorem{Expl}[Lem]{Example}

\newtheorem{Const}[Lem]{Construction}

\newenvironment{proof}[1][Proof]{\textrm{\em #1.} }{\hfill$\Box$\medskip\medskip}

\usepackage{titlesec}
\titleformat{\section}
{\centering\normalfont\scshape}{\thesection.}{0.5em}{}
\titleformat{\subsection}
{\normalfont}{\thesubsection.}{0.5em}{\textbf}

\newcommand\Gap{\operatorname{Gap}}
\newcommand\gap{\operatorname{-gap}}
\newcommand\wdt{{\operatorname{wd}}}
\newcommand\Shad{{\operatorname{Shad}}}
\newcommand\BShad{{\operatorname{BShad}}}

\newcommand\supp{{\operatorname{supp}}}

\newcommand\slex{{\operatorname{slex}}}

\newcommand\C{\mathcal{C}}
\def\NZQ{\mathbb}
\def\NN{{\NZQ N}}
\def\ZZ{{\NZQ Z}}
\newcommand\Corn{{\operatorname{Corn}}}

\let\emptyset\varnothing

\newcommand\blfootnote[1]{%
  \begingroup
  \renewcommand\thefootnote{}\footnote{#1}%
  \addtocounter{footnote}{-1}%
  \endgroup
}

\begin{document}
\title{\bf\normalsize\MakeUppercase{A numerical characterization of the extremal Betti numbers of} $t$--\MakeUppercase{spread strongly stable Ideals}}
\author{Luca Amata, Antonino Ficarra, Marilena Crupi}	

\newcommand{\Addresses}{{
  \footnotesize
  \textsc{Department of Mathematics and Computer Sciences, Physics and Earth Sciences, University of Messina, Viale Ferdinando Stagno d'Alcontres 31, 98166 Messina, Italy}
\begin{center}
 \textit{E-mail addresses}: \texttt{lamata@unime.it}; \texttt{antficarra@unime.it}; \texttt{mcrupi@unime.it}
\end{center}

}}
\date{}
\maketitle
\Addresses

\begin{abstract}
Let $K$ be a field and let $S=K[x_1,\dots,x_n]$ be a standard polynomial ring over a field $K$. We characterize the extremal Betti numbers, values as well positions, of a $t$--spread strongly stable ideal of $S$. 
Our approach is constructive.
Indeed, given some positive integers $a_1,\dots,a_r$ and some pairs of positive integers $(k_1,\ell_1),\dots,(k_r,\ell_r)$, we are able to determine under which conditions there exist a $t$--spread strongly stable ideal $I$ of $S$ with $\beta_{k_i, k_i\ell_i}(I)=a_i$, $i=1, \ldots, r$, as extremal Betti numbers, and then to construct it.


\blfootnote{
\hspace{-0,3cm} \emph{Keywords:} monomial ideals, minimal graded resolution, extremal Betti numbers, $t$-spread ideals.


}
\end{abstract}

\section{Introduction}\label{sec1}
Let $K$ be a field and let $S=K[x_1,\dots,x_n]$ be a polynomial ring. A squarefree monomial ideal of $S$, known also as Stanley--Reisner ideal, is a monomial ideal generated by squarefree monomials. The intimate relationship of this algebraic object with simplicial topology determines its combinatorial nature. The importance of such a class of monomial ideals  is due to Richard Stanley. In fact, in 1975,
Stanley used the theory of Cohen--Macaulay rings to solve the upper bound conjecture for spheres  \cite{RPS}. More in details, he used tools  from commutative algebra to study simplicial complexes by considering the Hilbert function of Stanley--Reisner
rings, whose defining ideals are squarefree ideals.

Recently, a generalization of the notion of squarefree ideal has been given  in \cite{EHQ} by the definition of $t$--spread monomial ideal.
Let $t\ge 0$ an integer and let $u=x_{i_1}x_{i_2}\cdots x_{i_d}$ be a monomial of $S$, with $1\le i_1\le i_2\le \dots \le i_d\le n$. $u$ is called $t$--spread if $i_{j+1}-i_j\ge t$, for all $j=1,\dots,d-1$. A monomial ideal $I$ is a $t$--spread monomial ideal if it is generated by $t$--spread monomials. Clearly, when $t=1$, a $1$--spread monomial ideal is a squarefree monomial ideal. Many results which are true for squarefree ideals can be suitably rewritten 
for the class of $t$--spread ideals. For example, the classical Kruskal--Katona theorem that characterizes the $f$--vectors of simplicial complexes has been generalized by C. Andrei--Ciobanu for the class of $t$--spread strongly stable ideals \cite{CAC}.

Among the algebraic invariants of a graded ideal $I$ of $S$, the graded Betti numbers $\beta_{i,j}(I)$ are probably the most important: the integer $\beta_{i,j}(I)$ gives the dimension, as a $K$--vector space, of the $j$--th component of the $i$--th free module in the minimal graded free resolution of $I$.  
They
are usually displayed in a table called the \textit{Betti table} of $I$: 
\begin{table}[H]
\begin{center}
\scalebox{0.85}{
\begin{tikzpicture}
\draw (0,4.5)--node[above]{$k$}(10,4.5);
\draw (1,0)--node[above left]{$\ell$}(1,5);
\draw [loosely dotted] (5,1.5)--(5,4.5);
\draw [loosely dotted] (1,3)--node[below right]{$\beta_{k,k + \ell}$}(8,3);
\draw [dashed] (1,0.8)--(4,0.8)--(4,1.5)--(7,1.5)--(7,2)--(8,2)--(8,4.5);
\draw [fill] (4,0.8) circle [radius=0.03] node [below]{$(k_3,\ell_3)$};
\draw [fill] (7,1.5) circle [radius=0.03] node [below]{$(k_2,\ell_2)$};
\draw [fill] (8,2) circle [radius=0.03] node [below right]{$(k_1,\ell_1)$};
\end{tikzpicture}
}
\end{center}
\caption{Betti table}\label{fig:carnumfig1}
\end{table}
\medskip
The columns are indexed from left to right by homological degree starting with homological degree
zero. The rows are indexed increasingly from top to bottom starting with the minimal degree of an element in a minimal system of homogeneous generators of $I$. The graded Betti numbers which appear in the outside corners of the dashed line are called extremal Betti numbers. This definition 
has been introduced by Bayer, Charalambous and Popescu in \cite{BCP}.

%
%

In this paper, our attention is devoted to $t$--spread strongly stable ideals (Definition \ref{def:strastbid}).  There is a combinatorial formula \cite[Corollary 1.12]{EHQ} that allows us to easily compute the graded Betti numbers of such a class of ideals. Moreover, thanks to this formula, one can give a characterization of these invariants (Theorem \ref{betti teor}) which is useful for the studying of their behavior. Such a result will be a pivotal tool throughout the paper.

Our goal is to solve the following question: \\
$(*)$ \emph{Given $n,r,t$ positive integers, $r$ positive integers $a_1,\dots,a_r$ and $r$ positive pairs of integers $(k_1,\ell_1),\dots,(k_r,\ell_r)$, under which conditions there exist a $t$--spread strongly stable ideal $I$ of $S=K[x_1,\dots,x_n]$ such that
$$
\beta_{k_1,k_1+\ell_1}(I)=a_1,\dots,\beta_{k_r,k_r+\ell_r}(I)=a_r
$$ 
are its extremal Betti numbers?}\medskip

For $t=1$, a positive answer can be found in \cite{AC1}; whereas in \cite{AFC} the maximal number $r$ of extremal Betti numbers allowed for a $t$--spread strongly stable ideals, for $t\ge 2$, has been computed. 
In this paper, we are able to give a positive answer to question (*) for $t\ge 2$.

Our approach is mainly constructive. We give a numerical characterization of the extremal Betti numbers of a $t$--spread strongly stable ideal and if the ideal that solves the problem does exist, then we provide the combinatorial tools to determine it.

The outline of the paper is the following. Section \ref{sec2}, contains preliminary
notions and results. In Section \ref{sec3}, we introduce the notion of $j\gap_t$ of a $t$--spread monomial  (Definition \ref{def:tgap}), $j$ is a positive integer, and the notion of Borel $t$--shadow of a set of $t$--spread monomials (Definition \ref{def:boshad}). These definitions will have a key role throughout the paper. 
Section \ref{sec4} contains a numerical characterization (positions and values) of the extremal Betti numbers of a $t$--strongly stable ideal (Theorem \ref{carnummainteor}). Finally,  Section \ref{sec:expl} contains an example which illustrates our methods.

All the examples have been verified using specific packages of \textit{Macaulay2} some of which developed by the authors of the paper. 

\section{Preliminaries}\label{sec2}

Let $S=K[x_1,\dots,x_n]$ be the standard polynomial ring in $n$ variables with coefficients in a field $K$. $S$ is an $\NN$--graded ring where $\deg x_i=1$, for all $i=1,\dots,n$. A monomial ideal $I$ of $S$ is an ideal generated by monomials. If $I$ is a monomial ideal, we denote by $G(I)$ the unique minimal set of monomial generators of $I$ ad we set
\[G(I)_{\ell}= \{u\in G(I)\,:\, \deg u=\ell\}.\]

For a monomial $u\in S$, $u\ne 1$, we denote by $\supp(u)$ the set of all index $i$ for which $x_i$ divides $u$ and by $\max(u)$ ($\min(u)$) the maximal (minimal) index $i$ for which $x_i$ divides $u$.

If $u=1$, for convenience, we set $\max(1)=\min(1)=0$.


The next definitions have been introduced in \cite{EHQ}.

\begin{Def}\rm
If $t\ge0$ is an integer, a monomial $x_{i_1}x_{i_2}\cdots x_{i_\ell}$ of $S$ with $1\le i_1\le i_2\le\ldots\le i_\ell\le n$ is called $t$--spread, if $i_{j+1}-i_j\ge t$, for all $j=1,\dots,\ell-1$.  A $t$--spread monomial ideal is a monomial ideal generated by $t$--spread monomials.
\end{Def}

For example, every monomial ideal of the polynomial ring $S$ is $0$--spread and every squarefree monomial ideal is a $1$--spread monomial ideal. Furthermore, one can observe that if $t\ge 1$, then every $t$--spread monomial is a squarefree monomial.\\

A special class of $t$--spread monomial ideals consists of the $t$--spread strongly stable ideals.

%
\begin{Def}\label{def:strastbid}\rm
A $t$--spread strongly stable ideal $I$ of $S$ is a $t$--spread monomial ideal such that for all $t$--spread monomials $u \in I$ and for all $i$, $j$ such that $1\le i < j\le n$, $x_j$ divides $u$ and $x_i(u/x_j)$ is $t$--spread, it follows that $x_i(u/x_j) \in I$.
\end{Def}

\begin{Rem}\rm
One can observe that the defining property of a $t$--strongly stable ideal needs to be checked only for the set of monomial generators of a $t$--spread monomial ideal \cite[Lemma 1.2]{EHQ}. In fact, let $I$ be a $t$--spread monomial ideal and suppose that for all $u\in G(I)$, for all integers $i<j$ with $j\in\supp(u)$ and such that $x_i(u/x_j)$ is a $t$--spread monomial, one has $x_i(u/x_j)\in I$. Then $I$ is a $t$--spread strongly stable ideal.
\end{Rem}

It is clear that the notion of $t$--spread strongly stable ideal generalizes the concept of strongly stable and squarefree strongly stable ideal.\\

There is an equivalent definition for a $t$--spread strongly stable ideal \cite[Definition 1.1]{CAC}. 
Following the same notations as in \cite{CAC}, we will denote by $[I_j]_t$ the set of all $t$--spread monomials of degree $j$ of an arbitrary monomial ideal $I$.

Let us denote by $M_{n,d,t}$ the set of all $t$--spread monomials of degree $d$ of the ring $S$.

\begin{Def}\rm
A subset $L$ of $M_{n,d,t}$ is called a $t$--spread strongly stable set, if for all $t$--spread monomials $u\in L$, all $j\in \supp(u)$ and all $1\le i< j$ such that $x_i(u/x_j)$ is a $t$--spread monomial, it follows that $x_i(u/x_j)\in L$.

A $t$--spread monomial ideal $I$ is a $t$--spread strongly stable ideal if $[I_j]_t$ is a $t$--spread strongly stable set for all $j$.
\end{Def}

Let $u_1,\dots,u_r$ be $t$--spread monomials of $S$. The unique $t$--spread strongly stable ideal containing $u_1,\dots,u_r$ will be denoted by $B_t(u_1,\dots,u_r)$, \cite{EHQ}. The monomials $u_1,\dots,u_r$ are called \textit{$t$--spread Borel generators}.

Furthermore, if $u_1,\dots,u_r$ are $t$--spread monomials of degree $d$,  the smallest $t$--spread strongly stable set of $M_{n,d, t}$ containing the monomials $u_1, \ldots, u_r$, will be denoted by $B_t\{ u_1, \ldots, u_r\}$.

Now, for a nonempty subset $L$ of $M_{n,d,t}$, we define the \textit{$t$--shadow} of $L$
$$
\Shad_t(L)= \big\{x_iw\,:\, w\in L\, \mbox{and $x_iw$ is $t$--spread monomial, $i=1,\dots,n$}\big\}.
$$

Throughout the paper, we assume that $t\ge 1$ and that  $M_{n,d,t}$  is endowed with the \emph{squarefree lexicographic order}, $>_{\slex}$ \cite{AHH2}, \emph{i.e.}, let $u=x_{i_1}x_{i_2}\cdots x_{i_d}$,  $v=x_{j_1}x_{j_2}\cdots x_{j_d}$,
be $t$--spread monomials of degree $d$, with $1\le i_1<i_2<\dots<i_d\le n$, and $1\le j_1<j_2<\dots<j_d\le n$, then
$u>_{\slex}v$ if $i_1=j_1,\dots,i_{s-1}=j_{s-1}$, and $i_s<j_s$,
for some $1\le s\le d$.

For a nonempty subset $T$  of $M_{n,d,t}$, we denote by $\max T$ ($\min T$) the maximal (minimum) monomial $w\in T$, with respect to $>_{\slex}$.\\




As for the class of (squarefree) strongly stable ideals of $S$, there exists an explicit formula due to Herzog, Ene and Qureshi  \cite[Corollary 1.12]{EHQ} for the graded Betti numbers of a $t$--spread strongly stable ideal given by
\begin{equation}\label{eqcarnum1}
\beta_{k,k+\ell}(I)=\sum_{u\in G(I)_\ell}\binom{\max(u)-t(\ell-1)-1}{k}.
\end{equation}

In order to apply such a formula, we just need to know the minimal set of monomial generators of $I$. 

\begin{Rem}\rm
If we put $t=0$ in (\ref{eqcarnum1}), then we obtain the well--known formula of Eliahou and Kervaire \cite{EK} for the (strongly) stable ideals; whereas, if we put $t=1$ in (\ref{eqcarnum1}), then we obtain the Aramova, Herzog and Hibi formula for squarefree (strongly) stable ideals \cite{AHH2}.
\end{Rem}

\begin{Def}\rm (\cite{BCP})
A graded Betti number $\beta_{k,k+\ell}(I)\ne 0$ is called \textit{extremal} if $\beta_{i,i+j}(I)=0$ for all $i\ge k,j\ge\ell,(i,j)\ne(k,\ell)$.
\end{Def}
	
If $\beta_{k,k+\ell}(I)$ is extremal, the pair $(k,\ell)$ is called a \textit{corner} of $I$.
If $(k_1,\ell_1),$ $\dots,$ $(k_r,\ell_r)$, with $n-1\ge k_1>k_2>\dots>k_r\ge1$ and $1\le \ell_1<\ell_2<\dots<\ell_r$, are all the corners of a graded ideal $I$ of $S$, the set
$$
\Corn(I)=\Big\{(k_1,\ell_1),(k_2,\ell_2),\dots,(k_r,\ell_r)\Big\}
$$
is called the \textit{corner sequence} of $I$ \cite{MC}. The $r$--tuple
$$
a(I)=\big(\beta_{k_1,k_1+\ell_1}(I),\beta_{k_2,k_2+\ell_2}(I),\dots,\beta_{k_r,k_r+\ell_r}(I)\big)
$$
is called the \textit{corner values sequence} of $I$ \cite{MC}.

As a consequence of formula (\ref{eqcarnum1}), in \cite{AC2} the following two results were stated. 
\begin{Thm} \label{betti teor} (\cite[Theorem 1]{AC2})
Let $I$ be a $t$--spread strongly stable ideal of $S$.
	
The following conditions are equivalent:
\begin{enumerate}
\item[\em(a)] $\beta_{k,k+\ell}(I)$ is extremal;
\item[\em(b)] $k+t(\ell-1)+1=\max\big\{\max(u):u\in G(I)_\ell\big\}$ and $\max(u)<k+t(j-1)+1$, for all $j>\ell$ and for all $u\in G(I)_j$.
\end{enumerate}
\end{Thm}

\begin{Cor}\label{AC1 cor2} (\cite[Corollary 2]{AC2})
Let $I$ be a $t$--spread strongly stable ideal of $S$ and let $\beta_{k,k+\ell}(I)$ be an extremal Betti number of $I$. Then
$$
\beta_{k,k+\ell}(I)=\Big|\Big\{ u\in G(I)_\ell:\max(u)=k+t(\ell-1)+1 \Big\}\Big|.
$$
\end{Cor}

We will see that Corollary \ref{AC1 cor2} will be fundamental in order to obtain the numerical characterization of the extremal Betti numbers of a $t$--spread strongly stable ideal.

\section{Combinatorics of $t$--spread monomials}\label{sec3}
	

In this Section, we generalize some notions and tools from \cite{AC1}.	
Our purpose is to suitably manipulate the $t$--spread monomials for $t\ge 1$.

Let $t\ge 1$ and $n$ be a positive integer. Let $(k,\ell)$ be a pair of positive integers such that $k+t(\ell-1)+1\le n$, we define the set
$$
A^t(k,\ell):=\big\{u\in M_{n,\ell,t}:\max(u)=k+t(\ell-1)+1\big\}.
$$
	
\begin{Rem}\rm
In \cite{AC1}, if $t=1$, the set $A^1(k,\ell)$ is denoted by $A^s(k,\ell)$ and consists of all squarefree monomials $u\in S$ of degree $\ell$ such that $\max(u) = k+\ell$.
\end{Rem}
	
\begin{Expl} \label{sec2:es1}\em Let $t=3$, $(k,\ell)=(5, 2)$ and $n\ge k+t(\ell-1)+1=9$. 

If $n=9$, then
\[
A^3(5,2)=\big\{u\in M_{9,2,3}:\max(u)=9\big\}
=\{x_1x_9,x_2x_9,x_3x_9,x_4x_9,x_5x_9,x_6x_9\}
\]
\end{Expl}
	
From Corollary \ref{AC1 cor2}, it is clear that in order to obtain a numerical characterization of the extremal Betti numbers of a $t$--spread strongly stable ideal $I$ of $S$ one has to consider the sets $A^t(k,\ell)$. 
	
Indeed, if $\beta_{k,k+\ell}(I)=a$ is an extremal Betti number of $I$, then
$$
a=\beta_{k,k+\ell}(I)=\Big|\Big\{ u\in G(I)_\ell:\max(u)=k+t(\ell-1)+1 \Big\}\Big|.
$$
	
Furthermore, since $$\big\{ u\in G(I)_\ell:\max(u)=k+t(\ell-1)+1\big\}\subseteq A^t(k,\ell),$$ then 
$$a= \beta_{k,k+\ell}(I)\le |A^t(k,\ell)|.$$
	
Hence, our first aim is to determine the cardinality of the sets $A^t(k,\ell)$.
	
From \cite[Theorem 2.3]{EHQ} (see also \cite{CAC}), the cardinality of $M_{n,d,t}$ is given by 
\begin{equation}\label{eqcarnumherz}
|M_{n,d,t}|=\binom{n-(d-1)(t-1)}{d}.
\end{equation}
	
\begin{Lem}\label{lemmacarnum1}
Let $k,\ell,t$ be three positive integers. Then $|A^t(k,\ell)|=\binom{k+\ell-1}{\ell-1}$. 
\end{Lem}
\begin{proof}
Let $u\in A^t(k,\ell)$, then $u=x_{i_1}x_{i_2}\cdots x_{i_{\ell}}$ is a $t$--spread momomial of degree $\ell$, with $\max(u)=k+(\ell-1)t+1$. Set $\widetilde{u}=x_{i_1}x_{i_2}\cdots x_{i_{\ell-1}}$, then $\widetilde{u}$ is a $t$--spread monomial of degree $\ell-1$ with $\max(\widetilde{u})\le\max(u)-t=k+(\ell-2)t+1$. Hence,
$$|A^t(k,\ell)| = \vert M_{k+(\ell-2)t+1,\ell-1,t}\vert.$$
Therefore, by (\ref{eqcarnumherz}), we have
\begin{eqnarray*}
|A^t(k,\ell)|=|M_{k+(\ell-2)t+1,\ell-1,t}|&=&\binom{k+(\ell-2)t+1-(\ell-1-1)(t-1)}{\ell-1}\\
&=& \binom{k+\ell-1}{\ell-1}.
\end{eqnarray*}
\end{proof}\medskip
	
For instance, if we consider Example \ref{sec2:es1}, then $A^3(5,2)$ consists of six  monomials. On the other hand, by Lemma \ref{lemmacarnum1}, $|A^3(5,2)|=\binom{5+2-1}{2-1}=6$.\\
	
Setting $A^t(k,\ell) = \{u_1,\dots,u_q\}$, $q=|A^t(k,\ell)|=\binom{k+\ell-1}{\ell-1}$, we can suppose, after a permutation of the indices, that  
\begin{equation}\label{eqcarnumor}
u_1>_{\slex}u_2>_{\slex}\dots>_{\slex}u_q.
\end{equation}
Thus, for the $i$--th ($1\le i\le q$) monomial $u$ of degree $\ell$ with $\max(u)= k+(\ell-1)t+1$, we mean the monomial of $A^t(k,\ell)$ that appears in the $i$--th position of (\ref{eqcarnumor}).
	
Now, let $(k,\ell)$ be a corner of a $t$--spread strongly stable ideal $I$ of $S$, and let $\beta_{k,k+\ell}(I)=a$. The bound $a\le|A^t(k,\ell)|$ can be improved. More precisely, let $u$ be the smallest $t$--spread monomial  belonging to $G(I)_\ell$ with $\max(u)=k+t(\ell-1)+1$, with respect to $>_{\slex}$. Thus, $u$ is equal to some $u_j\in A^t(k,\ell)$, $1\le j\le |A^t(k,\ell)|$. Therefore, by Corollary \ref{AC1 cor2},
\begin{equation}\label{sec3:eq1}
\begin{aligned}
a=\beta_{k,k+\ell}(I)=& \Big|\big\{ u\in G(I)_\ell:\max(u)=k+t(\ell-1)+1 \big\}\Big|\\
 \le & \Big|\big\{v\in A^t(k,\ell):v\ge_{\slex}u_j\big\}\Big|.
\end{aligned}
\end{equation}

	
\begin{Rem}\rm
A subset $L$ of $M_{n,d,t}$ is called $t$--spread lex set, if for all $u\in L$ and for all $v\in M_{n,d,t}$ with $v\ge_{\slex} u$, it follows that $v\in L$. We observe that in (\ref{sec3:eq1}), the equality  holds if and only if $G(I)_\ell$ is a $t$--spread lex set.
\end{Rem}

Now, we want to compute the cardinality of the set $\big\{v\in A^t(k,\ell):v\ge_{\slex}u_j\big\}$. 
	
To do this, we introduce the following subsets of $A^t(k,\ell)$. Let $u_i, u_j$, $i<j$, be two monomials of (\ref{eqcarnumor}), the sets
\begin{align*}
[u_i,u_j]&:=\big\{w\in A^t(k,\ell):u_i\ge_{\slex}w\ge_{\slex}u_j\big\},\\
[u_i,u_j)&:=\big\{w\in A^t(k,\ell):u_i\ge_{\slex}w>_{\slex}u_j\big\}
\end{align*}
are called \emph{segment of $A^t(k,\ell)$ of initial element $u_i$ and final element $u_j$} and \emph{left segment of $A^t(k,\ell)$ of initial element $u_i$ and final element $u_j$}, respectively.
	
Hence, the set $\big\{v\in A^t(k,\ell):v\ge_{\slex}u_j\big\}$ can be written as follow:
$$
[x_1x_{1+t}\dots x_{1+t(\ell-2)}x_{k+t(\ell-1)+1},u_j].
$$
Indeed, $x_1x_{1+t}\dots x_{1+t(\ell-2)}x_{k+t(\ell-1)+1}$ is the greatest monomial of $A^t(k,\ell)$ with respect to $>_{\slex}$.
	
In order to compute the cardinality of $[x_1x_{1+t}\dots x_{1+t(\ell-2)}x_{k+t(\ell-1)+1},u_j]$, we illustrate how one may decompose $|A^t(k,\ell)|$ \emph{via} $\min(u)$, for all $u\in A^t(k,\ell)$.
	
	
\begin{Prop}\cite[Lemma~4.3]{AC1}\label{propcarnum1}
Let $k,\ell$ be positive integers. Then
\begin{equation}\label{eqcarnumdecomp}
\begin{aligned}
\binom{k+\ell-1}{\ell-1}&=\sum_{i=1}^{k+1}\binom{k+\ell-1-i}{\ell-2}\\
&=\binom{k+\ell-2}{\ell-2}+\binom{k+\ell-3}{\ell-2}+\dots+\binom{\ell-2}{\ell-2}.
\end{aligned}
\end{equation}
\end{Prop}
	
%

\begin{Rem}\label{rem:carnum1} \em
It is important to discuss the binomial coefficients described in Proposition \ref{propcarnum1}. 

By Lemma \ref{lemmacarnum1}, $\binom{k+\ell-1}{\ell-1} = |A^t(k,\ell)|$. Therefore $\binom{k+\ell-1}{\ell-1}$ gives the number of $t$--spread monomials $u$ of degree $\ell$ with $\max(u)=k+t(\ell-1)+1$. 

Note that for a monomial $u\in A^t(k,\ell)$ one has $\min(u)\le k+1$. Indeed, the smallest monomial of $A^t(k,\ell)$ with respect to $>_{\slex}$ is
$$x_{k+1}x_{k+t+1}x_{k+2t+1}\cdots x_{k+(\ell-1)t+1}.$$
Hence
$$
\big|A^t(k,\ell)\big|=\binom{k+\ell-1}{\ell-1}=\sum_{i=1}^{k+1}b_i,
$$
where $b_i=\left|\big\{u\in A^t(k,\ell):\min(u)=i\big\}\right|$, for all $i=1,\dots,k+1$.
		
Now, fix $i\in\{1,\dots,k+1\}$. The integer $b_i$ is the number of all monomials $u$ of $A^t(k,\ell)$ with $\min(u)=i$. If $u$ is a monomial of this type, then $u$ can be written as $u=x_{i_1}x_{i_2}\cdots x_{i_\ell}$, with $i_1=i$, $i_\ell=k+t(\ell-1)+1$ and $i_2\ge i_1+t=i+t$. Setting $\widetilde{u}=x_{i_2-(i+t)+1}x_{i_3-(i+t)+1}\cdots x_{i_\ell-(i+t)+1}$, then $i_2-(i+t)+1\ge 1$. Therefore, computing the number of all monomials $u$ of $A^t(k,\ell)$ with $\min(u)=i$ is equivalent to determining the number of all $t$--spread monomials $\widetilde{u}$ of degree $\ell-1$ with $\min(u)\ge 1$ and $\max(u)=i_\ell-(i+t)+1=(k-i+1)+(\ell-2)t+1$; such a number is the cardinality of the set $A^t(k-i+1,\ell-1)$. Therefore
$$
b_i=\big|A^t(k-i+1,\ell-1)\big|=\binom{k-i+1+\ell-1-1}{\ell-1-1}=\binom{k+\ell-1-i}{\ell-2}.
$$
\end{Rem}
	

For a positive integer $q$, we set $[q] = \{1,\dots,q\}$. Following \cite{AC1}, we introduce the following definition that will have a key role throughout the paper.

\begin{Def}\label{def:tgap} \rm
Let $u=x_{i_1}x_{i_2}\cdots x_{i_d}$ be a $t$--spread monomial of degree $d$ of $S$, $1\le i_1<i_2<\ldots<i_d\le n$. Let $j\in[d-1]$. We say that $u$ has a \textit{$j\gap_t$} if
$$
i_{j+1}-i_j-t>0.
$$
If $u$ has a \textit{$j\gap_t$}, then we set 
$$\wdt(j\gap_t)(u):=i_{j+1}-i_j-t$$ and we call it the \textit{width of the $j\gap_t$}. 
\end{Def}
	
We define $\Gap_t(u)$ by
$$
\Gap_t(u):=\big\{j\in[d-1]:\text{there exists a}\ j\gap_t(u)\big\}.
$$
The following example illustrates the \emph{combinatorial meaning} of the width of a $j\gap_t$.
\begin{Expl}\rm
Let $t=2$ and $u=x_{i_1}x_{i_2}x_{i_3}x_{i_4}x_{i_5}=x_2x_4x_6x_{13}x_{15}\in$ $K[x_1, \ldots, x_{15}]$ be a $2$--spread monomial of degree $5$. We have $\Gap_2(u)=\{3\}$ and $\wdt(3\gap_2)$ $(u)=13-6-2=5$. The width has the following meaning: \emph{there exist exactly five variables} (\emph{each of them has index smaller than $13$}) \emph{that may take the place of $x_{i_4}=x_{13}$ in $u$ in order to still obtain a $2$--spread monomial}. These variables are $x_{12},x_{11},x_{10},x_9,x_8$. Indeed,
\begin{align*}
x_{12}(u/x_{13})&=x_2x_4x_6x_{12}x_{15}, & &\\
x_{11}(u/x_{13})&=x_2x_4x_6x_{11}x_{15}, & x_9(u/x_{13})&=x_2x_4x_6x_9x_{15},\\
x_{10}(u/x_{13})&=x_2x_4x_6x_{10}x_{15}, & x_8(u/x_{13})&=x_2x_4x_6x_8x_{15},
\end{align*}
are all $2$--spread monomials.
\end{Expl}\medskip
	
There is another reason to consider the $j\gap_t$ of a $t$--spread monomial $u\in A^t(k,\ell)$, the set $\Gap_t(u)$ allow us to find the largest $t$--spread monomial of $A^t(k,\ell)$ that follows $u$ in the order $>_{\slex}$, 
as the next result shows.
	
\begin{Prop}
Let $n,t$ be positive integers. Let $(k,\ell)$ be a pair of positive integers such that $k+t(\ell-1)+1\le n$. Let $u=x_{i_1}x_{i_2}\cdots x_{i_\ell}$, $1\le i_1<i_2<\dots<i_\ell\le n$ be a $t$--spread monomial of $A^t(k,\ell)$, that is $i_\ell=k+t(\ell-1)+1$. If $\Gap_t(u)=\emptyset$, then $u$ is the smallest $t$--spread monomial of $A^t(k,\ell)$ with respect to $>_{\slex}$. Otherwise, if $p=\max\Gap_t(u)$, then the greatest  monomial of $A^t(k,\ell)$ following $u$ in the squarefree lexicographic order is
$$
x_{i_1}\cdots x_{i_{p-1}}x_{i_p+1}x_{i_p+1+t}\cdots x_{i_p+1+t(\ell-p-1)}x_{k+t(\ell-1)+1}.
$$
\end{Prop}
\begin{proof} The proof is similar to the proof of \cite[Lemma 3.1]{AFC}. We include it for the reader's convenience.
	
If $\Gap_t(u)=\emptyset$, then $u=x_kx_{k+t+1}\cdots x_{k+t(\ell-2)+1}x_{k+t(\ell-1)+1}$. Hence, $u$ is the smallest $t$--spread monomial of degree $d$.
		
Now, let $\Gap_t(u)\ne\emptyset$ and let $p=\max\Gap_t(u)$.
If $w=x_{s_1}x_{s_2}\cdots x_{s_\ell}$ is a monomial of $A^t(k,\ell)$ with $u>_{\slex}w$, then  $i_1=s_1,\dots,i_{j-1}=s_{j-1}$ and $i_j<s_j$ for some index $j$. We have $j\le p$. Indeed, if $j>p$, then $i_{p+2}-i_{p+1}=\dots=i_\ell-i_{\ell-1}=t$. Hence, $i_{\ell}-i_j=t(\ell-j)$, $i_\ell=s_\ell=k+t(\ell-1)+1$, $s_{\ell}-s_j\ge t(\ell-j)$ and $i_j<s_j$. Therefore, we have
$$
t(\ell-j)\le s_{\ell}-s_j=i_\ell-s_j<i_\ell-i_j=t(\ell-j),
$$
and consequently $t(\ell-j)<t(\ell-j)$. This is absurd. So $j\le p$.
		
Let
$$
v=x_{j_1}x_{j_2}\cdots x_{j_d}=x_{i_1}\cdots x_{i_{p-1}}x_{i_p+1}x_{i_p+1+t}\cdots x_{i_p+1+t(\ell-p-1)}x_{k+t(\ell-1)+1}.
$$
	
Since $i_1=j_1,\dots,i_{p-1}=j_{p-1}$ and $i_p<j_p=i_p+1$, then $u>_{\slex}v$. Clearly, $v$ is the greatest monomial of $A^t(k,\ell)$ following $u$ in the squarefree lexicographic order.
\end{proof}
\vspace{0,3cm}
	
%
%
%
For istance, if $u=x_{i_1}x_{i_2}x_{i_3}x_{i_4}x_{i_5}=x_2x_6x_{11}x_{14}x_{17}\in A^3(4,5)=A^t(k,\ell)$, then $p=\max\Gap_3(u)=2$, and the greatest monomial $v$ of $A^3(4,5)$ that follows $u$, with respect to $>_{\slex}$, is
\begin{align*}
v&=x_{i_1}\cdots x_{i_{p-1}}x_{i_p+1}x_{i_p+1+t}\cdots x_{i_p+1+t(\ell-p-1)}x_{k+t(\ell-1)+1}\\
&=x_{i_1}x_{i_2+1}x_{i_2+1+t}x_{i_2+1+2t}x_{k+t(\ell-1)+1}\\
&=x_2x_7x_{10}x_{13}x_{17}.
\end{align*}
\vspace{0,2cm}
	
	
The following result shows that the cardinality of the segment
$$[x_1x_{1+t}\cdots x_{1+t(\ell-2)}x_{k+t(\ell-1)+1},u]$$
depends on the structure of the monomial $u$. More precisely, in order to compute $|[x_1x_{1+t}\cdots x_{1+t(\ell-2)}x_{k+t(\ell-1)+1},u]|$, one needs to analyze the set $\Gap_t(u/x_{\max(u)})$. 

\begin{Thm}\label{teorcarnum1}
Let $(k,\ell)$ be a pair of positive integers with $\ell\ge2$ and let $u=x_{i_1}x_{i_2}\cdots $ $x_{i_{\ell-1}}x_{i_\ell}$ be a monomial of $A^t(k,\ell)$, $t\ge1$. 
		
Setting $\widetilde{u}=x_{i_1}x_{i_2}\cdots x_{i_{\ell-1}}$, then $|[x_1x_{1+t}\cdots x_{1+t(\ell-2)}x_{k+t(\ell-1)+1},u]|$ is a sum of $r$ suitable binomial coefficients, where
		
$$
r=\begin{cases}
i_1,&\text{if}\ \Gap_t(\widetilde{u})=\emptyset,\\
i_1+\sum_{j=1}^p\wdt(g_j\gap_t)(\widetilde{u}),&\text{if}\ \Gap_t(\widetilde{u})=\{g_1,\dots,g_p\}\ne\emptyset.
\end{cases}
$$
\end{Thm}
\begin{proof}
Let $m=|[x_1x_{1+t}\cdots x_{1+t(\ell-2)}x_{k+t(\ell-1)+1},u]|$, \emph{i.e.}, $m$ is the number of all monomials $w\in A^t(k,\ell)$ such that $w\ge_{\slex}u$. By Proposition \ref{propcarnum1}, the binomial coefficient $\binom{k+\ell-1}{\ell-1}$ can be decomposed as a sum of $k+1$ binomial coefficients:
\begin{equation}
\begin{aligned}\label{eqcarnum2}
\binom{k+\ell-1}{\ell-1}&=\sum_{j=1}^{k+1}\binom{k+\ell-1-j}{\ell-2}\\
&=\binom{k+\ell-2}{\ell-2}+\binom{k+\ell-3}{\ell-2}+\dots+\binom{\ell-2}{\ell-2}.
\end{aligned}
\end{equation}
		
From Remark \ref{rem:carnum1}, for all $j=1,\dots,k+1$, the $j$--th binomial coefficient of (\ref{eqcarnum2}), $\binom{k+\ell-1-j}{\ell-2}$, counts the number of $t$--spread monomials $w$ of degree $\ell$ with $\min(w)=j$ and $\max(w)=k+t(\ell-1)+1$. Now, every monomial $w\in A^t(k,\ell)$ with $\min(w)<i_1=\min(\widetilde{u})=\min(u)$ is greater than $u$ with respect to $>_{\slex}$. Therefore the first $i_1-1$ binomial coefficients in (\ref{eqcarnum2}) give a contribution to the computation of $m$.\medskip
		
We need to distinguish two cases: $\Gap_t(\widetilde{u})=\emptyset$, $\Gap_t(\widetilde{u})\ne\emptyset$.\\ \\
\textbf{Case 1.} Let $\Gap_t(\widetilde{u})=\emptyset$. In such a case, $u$ is the greatest monomial of $A^t(k,\ell)$ with $\min(u)=i_1$. Therefore, the following sum of binomial coefficients
\begin{equation}\label{eqcarnum3}
\sum_{j=1}^{i_1-1}\binom{k+\ell-1-j}{\ell-2}
\end{equation}
gives the number of all monomials $w\in A^t(k,\ell)$ greater than $u$. In fact, since $i_{j+1}-i_j=t$, for all $j=1,\dots,\ell-1$, then the other monomials greater than $u$ which are different from the $w$'s counted by (\ref{eqcarnum3}) do not exist. Hence,
\begin{align*}
m=\big|[x_1x_{1+t}\cdots x_{1+t(\ell-2)}x_{k+t(\ell-1)+1},u]\big|&=\sum_{j=1}^{i_1-1}\binom{k+\ell-1-j}{\ell-2}+1\\&=\sum_{j=1}^{i_1-1}\binom{k+\ell-1-j}{\ell-2}+\binom{0}{0}.
\end{align*}
Therefore, $m$ is the sum of $r=i_1-1+1=i_1=\min(\widetilde{u})=\min(u)$ binomial coefficients.\medskip\\
\textbf{Case 2.} Assume $\Gap_t(\widetilde{u})=\{g_1,\dots,g_p\}\ne\emptyset, p\ge1$.
Let $w\in A^t(k,\ell)$ such that $w=x_{j_1}x_{j_2}\cdots x_{j_\ell}$, $w>_{\slex}u$. Since $i_\ell=j_\ell$, then there exists an integer $s\in\{1,\dots,\ell-1\}$ such that
$$
j_1=i_1,\ \ \ j_2=i_2, \dots, j_{s-1}=i_{s-1},\ \ \ j_s<i_s.
$$
If $s=1$, then $w$ belongs to the set of monomials counted by (\ref{eqcarnum3}). Let $s>1$, then $s-1\in\Gap_t(\widetilde{u})$. Indeed,
$$
i_{s}-i_{s-1}-t>j_s-j_{s-1}-t=\wdt((s-1)\gap_t)(w)\ge0,
$$
and so $\wdt((s-1)\gap_t)(\widetilde{u})>0$. Therefore $s-1=g_h$ for some $h\in\{1,\dots,p\}$. Since, $j_s<i_s$, then $j_s\le i_s-1=i_{(g_h+1)}-1$ and $w=x_{i_1}x_{i_2}\cdots x_{i_{g_h}}z$, with $z$ a $t$--spread monomial of degree $\ell-g_h$, with
\begin{align*}
\max(z)&=k+t(\ell-1)+1,\\ \min(z)&=j_s\le i_s-1=i_{(g_h+1)}-1, 
\end{align*}
and, furthermore, $\min(z)\ge i_{g_h}+t$. Thus, $\min(z)$ belongs to the set
\small
\begin{equation}\label{eq:widegap}
\begin{aligned}
\big\{i_{g_h}+t,\dots,i_{(g_{h}+1)}-1\big\}&=\big\{i_{g_h}+t,\dots,i_{(g_{h}+1)}-i_{g_h}-t+i_{g_h}+t-1\big\}\\
&=\big\{i_{g_h}+t,\dots,i_{g_h}+t+\wdt(h\gap_t)(\widetilde{u})-1\big\}.
\end{aligned}
\end{equation}
\normalsize
		
It follows that $\min(z)$ may assume the $\wdt(h\gap_t)(\widetilde{u})$ possible values described in (\ref{eq:widegap}).\\
Now, let us consider the $i_1$--th binomial in (\ref{eqcarnum2}):
\begin{equation}\label{eqcarnum4}
\binom{k+\ell-1-i_1}{\ell-2}=\sum_{j=1}^{k-i_1+2}\binom{k+\ell-1-i_1-j}{\ell-3}.
\end{equation}
The $j$--th binomial coefficient of this sum counts the number of all monomials $w =x_{j_1}x_{j_2}\cdots x_{j_\ell}$ $\in A^t(k,\ell)$ such that $\min(w)=j_1=i_1$ and $j_2=i_2+t+j-1$. Hence, if $1\in\Gap_t(\widetilde{u})$, by the same arguments as before, the sum of the first $\wdt(1\gap_t)(\widetilde{u})$ in (\ref{eqcarnum4}) gives the number of monomials $w=x_{i_1}z\in A^t(k,\ell)$ greater than $u$. Thus, we consider the $\big(\wdt(1\gap_t)(\widetilde{u})+1\big)$--th binomial coefficient and we decompose it; if $2\in\Gap_t(\widetilde{u})$, then the sum of the first $\wdt(2\gap_t)(\widetilde{u})$ binomial coefficients of this new binomial decomposition gives the number of monomials $x_{i_1}x_{i_2}z\in A^t(k,\ell)$ greater than $u$.\\
Otherwise, if $1\notin\Gap_t(\widetilde{u})$, let us consider the first binomial in (\ref{eqcarnum4}):
$$
\binom{k+\ell-2-i_1}{\ell-3}=\sum_{j=1}^{k-i_1+2}\binom{k+\ell-2-i_1-j}{\ell-4}.
$$
The $j$--th binomial coefficient of this sum counts the number of all monomials  $w=x_{j_1}x_{j_2}\cdots x_{j_\ell}$ $\in A^t(k,\ell)$ with $\min(w)=j_1=i_1$, $j_2=i_2$ and $j_3=i_3+t+j-1$. If $2\notin\Gap_t(\widetilde{u})$, then we consider the first binomial of this sum and we decompose it. We continue to evaluate successive binomial decomposition until we reach a  gap $g_j\in\Gap_t(\widetilde{u})$, then the first $\wdt(g_j\gap_t)(\widetilde{u})$ binomial coefficients of the last evaluated binomial decomposition will give a contribute to the computation of $m$. This procedure can be iterated for all $g_j\in\Gap_t(\widetilde{u})$, $j\ge 3$.\\		
So, $|[x_1x_{1+t}\cdots x_{1+t(\ell-2)}x_{k+t(\ell-1)+1},u)|$ is the sum of
$$i_1-1+\sum_{j=1}^p\wdt(g_j\gap_t)(\widetilde{u})$$
suitable binomial coefficients. 
		
Finally, in order to get $m=|[x_1x_{1+t}\cdots x_{1+t(\ell-2)}x_{k+t(\ell-1)+1},u]|$,
we must take into account the binomial $\binom{0}{0}=1$ which counts the monomial $u$:
$$
r=i_1-1+\sum_{j=1}^p\wdt(g_j\gap_t)(\widetilde{u})+1=i_1+\sum_{j=1}^p\wdt(g_j\gap_t)(\widetilde{u}).
$$
The assertion follows.
\end{proof}\medskip
	
The next example illustrates Theorem \ref{teorcarnum1}.
\begin{Expl}\rm
Let $S=K[x_1,\dots,x_{16}]$, $t=3,(k,\ell)=(6,4)$, then $k+t(\ell-1)+1=16$. Let $u=x_4x_9x_{13}x_{16}\in A^3(6,4)$. The greatest monomial of $A^3(6,4)$, with respect to $>_{\slex}$, is $v=x_1x_4x_7x_{16}$. We want to evaluate $m=|[v,u]|$. 
		
From Lemma \ref{lemmacarnum1}, $|A^3(6,4)|=\binom{9}{3}=84$.\\
Using the notation as in Theorem \ref{teorcarnum1}, $\widetilde{u}=x_4x_9x_{13}$, $\Gap_3(\widetilde{u})=\{1,2\}$ and
\begin{align*}
i_1&=4,\\
\wdt(1\gap_3)(\widetilde{u})&=9-4-3=2,\\
\wdt(2\gap_3)(\widetilde{u})&=13-9-3=1.
\end{align*}
		
In order to compute $m=|[x_1x_4x_7x_{16},u]|$, we consider the following binomial decomposition (Proposition \ref{propcarnum1}):
\begin{equation}\label{eq1:bin}
\binom{9}{3}=\binom{8}{2}+\binom{7}{2}+\binom{6}{2}+\binom{5}{2}+\binom{4}{2}+\binom{3}{2}+\binom{2}{2}.
\end{equation}

Since $\min(u)=\min(\widetilde{u})=i_1=4$, then all monomials $w\in A^3(6,4)$ with $\min(w)\le i_1-1=3$ are greater than $u$. Hence, for the computation of $m$ we must take into account the sum of the first three binomial coefficients in (\ref{eq1:bin}), \emph{i.e.}, $\boxed{ \tbinom{8}{2}+\tbinom{7}{2}+\tbinom{6}{2}=64}$.
		
Now, let us consider the following binomial decompositition:
$$
\binom{5}{2}=\binom{4}{1}+\binom{3}{1}+\binom{2}{1}+\binom{1}{1}.
$$		
Since $\wdt(1\gap_3)(\widetilde{u})=2$, the sum $\boxed{\tbinom{4}{1}+\tbinom{3}{1}=7}$ gives the number of all monomials of type $x_{i_1}z=x_4z\in A^3(6,4)$, with $z$ a $t$--spread monomial of degree 3 such that $\max(z)=16$ and $\min(z)\in\{7,8\}$.\\
Up to this point we have got $\boxed{64+7=71}$ monomials.
	
The next decomposition we need to consider is the following one
$$
\binom{2}{1}=\binom{1}{0}+\binom{0}{0}.
$$	
Since $2\in\Gap_3(\widetilde{u})$ and $\wdt(2\gap_3)(\widetilde{u})=1$, we must take into account $\boxed{\tbinom{1}{0}=1}$ monomial of the type $x_{i_1}x_{i_2}z=x_4x_9z\in A^3(6,4)$ with $\max(z)=16$ and $\min(z)=12$. 
		
Finally, we have obtained $\boxed{71+1=72}$ monomials of $A^3(6,4)$ greater than $u$, and so $m=|[x_1x_4x_7x_{16},u]|=72+1=73$.
		
The following scheme summarizes the previous calculation.
\begin{align*}
\tbinom{9}{3}=\boxed{\bf{\tbinom{8}{2}+\tbinom{7}{2}+\tbinom{6}{2}}} + & \tbinom{5}{2}+\tbinom{4}{2}+\tbinom{3}{2}+\tbinom{2}{2}\\
& \tbinom{5}{2}=
\begin{aligned}[t]  \boxed{\bf{\tbinom{4}{1}+\tbinom{3}{1}}}+\tbinom{2}{1}&+\tbinom{1}{1} \\
\tbinom{2}{1}&=\boxed{\bf{\tbinom{1}{0}}}+\tbinom{0}{0}.
\end{aligned}
\end{align*}
		
		
Figure \ref{fig:carnumfig2} shows the list of all monomials which come into play to determine $[v,u]$:
\begin{figure}[H]
\scalebox{0.8}{
$\begin{aligned}
\left. \begin{array}{r}
x_1x_4x_7x_{16},x_1x_4x_8x_{16},x_1x_4x_9x_{16},x_1x_4x_{10}x_{16},x_1x_4x_{11}x_{16},x_1x_4x_{12}x_{16},x_1x_4x_{13}x_{16}\\
x_1x_5x_8x_{16},x_1x_5x_9x_{16},x_1x_5x_{10}x_{16},x_1x_5x_{11}x_{16},x_1x_5x_{12}x_{16},x_1x_5x_{13}x_{16}\\
x_1x_6x_9x_{16},x_1x_6x_{10}x_{16},x_1x_6x_{11}x_{16},x_1x_6x_{12}x_{16},x_1x_6x_{13}x_{16}\\
x_1x_7x_{10}x_{16},x_1x_7x_{11}x_{16},x_1x_7x_{12}x_{16},x_1x_7x_{13}x_{16}\\
x_1x_8x_{11}x_{16},x_1x_8x_{12}x_{16},x_1x_8x_{13}x_{16}\\
x_1x_9x_{12}x_{16},x_1x_9x_{13}x_{16}\\
x_1x_{10}x_{13}x_{16}\\
\end{array}\right\}&\bf{\tbinom{8}{2}}\\
\left.\begin{array}{r}
x_2x_5x_8x_{16},x_2x_5x_9x_{16},x_2x_5x_{10}x_{16},x_2x_5x_{11}x_{16},x_2x_5x_{12}x_{16},x_2x_5x_{13}x_{16}\\
x_2x_6x_9x_{16},x_2x_6x_{10}x_{16},x_2x_6x_{11}x_{16},x_2x_6x_{12}x_{16},x_2x_6x_{13}x_{16}\\
x_2x_7x_{10}x_{16},x_2x_7x_{11}x_{16},x_2x_7x_{12}x_{16},x_2x_7x_{13}x_{16}\\
x_2x_8x_{11}x_{16},x_2x_8x_{12}x_{16},x_2x_8x_{13}x_{16}\\
x_2x_9x_{12}x_{16},x_2x_9x_{13}x_{16}\\
x_2x_{10}x_{13}x_{16}\\
\end{array}\right\}&\bf{\tbinom{7}{2}}\\
\left.\begin{array}{r}
x_3x_6x_9x_{16},x_3x_6x_{10}x_{16},x_3x_6x_{11}x_{16},x_3x_6x_{12}x_{16},x_3x_6x_{13}x_{16}\\
x_3x_7x_{10}x_{16},x_3x_7x_{11}x_{16},x_3x_7x_{12}x_{16},x_3x_7x_{13}x_{16}\\
x_3x_8x_{11}x_{16},x_3x_8x_{12}x_{16},x_3x_8x_{13}x_{16}\\
x_3x_9x_{12}x_{16},x_3x_9x_{13}x_{16}\\
x_3x_{10}x_{13}x_{16}\\
\end{array}\right\}&\bf{\tbinom{6}{2}}\\
\left.\begin{array}{r}
x_4x_7x_{10}x_{16},x_4x_7x_{11}x_{16},x_4x_7x_{12}x_{16},x_4x_7x_{13}x_{16}\\
x_4x_8x_{11}x_{16},x_4x_8x_{12}x_{16},x_4x_8x_{13}x_{16}
\end{array}\right\}&\bf{\tbinom{4}{1}+\tbinom{3}{1}}\\
\left.\begin{array}{r}
x_4x_9x_{12}x_{16}
\end{array}\right\}&\bf{\tbinom{1}{0}}\\
\left.\begin{array}{r}
\bm{x_4x_9x_{13}x_{16}}
\end{array}\right\}&\bf{\tbinom{0}{0}}
\end{aligned}$
}
\caption{Monomials of $[v,u]$.}
\label{fig:carnumfig2}
\end{figure}
\end{Expl}
\medskip
	
\begin{Rem} \em 
Let $I$ be a $t$--spread strongly stable ideal of $S=K[x_1,\dots,x_n]$ generated in degree $\ell$, and let $k$ and $a$ two positive integers. Then, $I$ has extremal Betti number $\beta_{k,k+\ell}(I)=a$ if and only if $k+t(\ell-1)+1\le n$ and $a\le|A^t(k,\ell)|$. 
 
If $A^t(k,\ell)=\{u_1,\dots,u_q\}$, 
with $u_1,\dots,u_q$ as in (\ref{eqcarnumor}), and $u_j$ is the smallest monomial of $A^t(k,\ell)$ belonging to $G(I)_\ell$, then
\begin{align*}
a &\le |\{v\in A^t(k,\ell):v\ge_{\slex}u_j\}|\\
  &= |[x_1x_{1+t}\cdots x_{1+t(\ell-2)+1}x_{k+t(\ell-1)+1},u_j]|=m,
\end{align*}
and Theorem \ref{teorcarnum1} allows to determine $m$. Therefore, if the ideal $I$ is generated in one degree $\ell$, then the proposed question $(*)$ (see, the section Introduction) is solved.
\end{Rem}
	
If $I$ is a $t$--spread strongly stable ideal generated in several degrees and $\beta_{k,k+\ell}(I)=a$ is an extremal Betti number, the bound $a\le|A^t(k,\ell)|$ is not sufficient to solve the question $(*)$, as next example shows.
	
\begin{Expl}\label{carnumes1}\rm
Let $n=13$, $t=2$. We will show that a $t$--spread strongly stable ideal $I$ of $S=K[x_1,\dots,x_{13}]$ with corners $(k_1,\ell_1)=(5,2)$, $(k_2,\ell_2)=(3,4)$, $k_1>k_2,\ 2\le\ell_1<\ell_2, k_i+t(\ell_i-1)+1\le 13$ ($i=1,2$) and extremal Betti numbers $\beta_{k_1,k_1+\ell_1}(I)=a_1=3$ and $\beta_{k_2,k_2+\ell_2}(I)=a_2=10$ does not exist.
		
First of all, we write down the monomials of the sets $A^t(k_i,\ell_i)$, $i=1,2$:
\begin{align*}
A^2(5,2)=\big\{&x_1x_8, x_2x_8,x_3x_8, x_4x_8, x_5x_8, x_6x_8 \big\},\\
A^2(3,4)=\big\{&x_1x_3x_5x_{10}, x_1x_3x_6x_{10}, x_1x_3x_7x_{10}, x_1x_3x_8x_{10}, x_1x_4x_6x_{10},\\ 
&x_1x_4x_7x_{10}, x_1x_4x_8x_{10}, x_1x_5x_7x_{10}, x_1x_5x_8x_{10}, x_1x_6x_8x_{10},\\
&x_2x_4x_6x_{10}, x_2x_4x_7x_{10}, x_2x_4x_8x_{10}, x_2x_5x_7x_{10}, x_2x_5x_8x_{10},\\
&x_2x_6x_8x_{10}, x_3x_5x_7x_{10}, x_3x_5x_8x_{10}, x_3x_6x_8x_{10}, x_4x_6x_8x_{10}\big\}.
\end{align*}

Assume that the sets $A^t(k_i,\ell_i)$ are ordered by $>_\slex$.
		
We observe that for the $a_i$'s ($i=1,2$) the following upper bounds hold
\begin{align*}
a_1\le& |A^t(k_1,\ell_1)|=\binom{5+2-1}{2-1}=6, \\
a_2\le& |A^t(k_2,\ell_2)|=\binom{3+4-1}{4-1}=20.
\end{align*} 
		
From Corollary \ref{AC1 cor2},
$$
\beta_{k_1,k_1+\ell_1}(I)=\beta_{5,7}(I)=\Big|\Big\{ u\in G(I)_{2}:\max(u)=8 \Big\}\Big|.
$$
$G(I)_{\ell_1}= G(I)_2$ has to be a $2$--spread strongly stable set such that $|G(I)_2\cap A^2(5,2)|=a_1$ $=3$.		
Hence,
$G(I)_2\cap A^2(5,2)=\{x_1x_8,x_2x_8,x_3x_8\}$. Therefore we set,
$$
G(I)_2=B_2\{ x_1x_8,x_2x_8,x_3x_8\}.
$$
		
Now, we have to determine $G(I)_{\ell_2}=G(I)_4$ such that $\beta_{k_2,k_2+\ell_2}(I)=a_2=10$. So, we have to find $10$ monomials of $A^2(3,4)$ not belonging to $\Shad_2^{\ell_2-\ell_1}(G(I)_{\ell_1})$.  
Since $\min\Shad_2^{\ell_2-\ell_1}(G(I)_{\ell_1}) = x_3x_6x_8x_{10}$, then 
the only monomial of $A^t(k_2,\ell_2)=A^2(3,4)$ that may belong to $G(I)_{\ell_2}=G(I)_4$ is $x_4x_6x_8x_{10}$. Hence, the value of $\beta_{k_2,k_2+\ell_2}(I)=\beta_{3,7}(I)$ can be at most $1$. Thus, the value $a_{2}=10$ cannot be reached.

Finally, a $2$--spread strongly stable ideal $I$ of $S$ with extremal Betti numbers $\beta_{k_1,k_1+\ell_1}(I)=a_1=3$ and $\beta_{k_2,k_2+\ell_2}(I)=a_2=10$ does not exists.
\end{Expl}\medskip

We close the Section by the next definition that will be useful for solving our problem.
\begin{Def} \label{def:boshad}\rm 
Let $(k_1,\ell_1)$, $(k_2,\ell_2)$ be two pairs of positive integers such that $k_1>k_2,\ell_1<\ell_2$, and $k_i+t(\ell_i-1)+1\le n$, $i=1,2$. If $T=\{u_1,\dots,u_r\}$ is a subset of monomials of $A^t(k_1,\ell_1)$, we define the \textit{Borel $t$--shadow of $T=\{u_1,\dots,u_r\}$ with respect to the corner $(k_2,\ell_2)$} as follows\\
\[
\BShad_t(T)_{(k_2,\ell_2)}=\Big\{ v\in\Shad_t^{\ell_2-\ell_1}\big(B_t\{ T\}\big):\max(v)\le k_2+t(\ell_2-1)+1 \Big\}.
\]
\normalsize
\end{Def}
	
\begin{Rem}\label{carnumoss1}\rm
It is clear that $\BShad_t(u_1,\dots,u_r)_{(k_2,\ell_2)}$ is a $t$--spread strongly stable set. Moreover, one can easily verify that,  if $u_1>_\slex u_2>_\slex \ldots >_\slex u_r$, then $$\min\BShad_t(u_1,\dots,u_r)_{(k_2,\ell_2)}=\min\BShad_t(u_{r})_{(k_2,\ell_2)}.$$
\end{Rem}\medskip

\section{A numerical characterization}\label{sec4}
In this Section we give a numerical characterization of the extremal Betti numbers of a $t$--spread strongly stable ideal by generalizing the methods in \cite{AC1}.

In order to find optimal bounds for the values of the extremal Betti numbers of a $t$--spread strongly stable ideal, we need some comments.
\begin{Discus}\label{carnumdisc1}\rm
Let $t\ge1$ and $S=K[x_1,\dots,x_n]$. Let $(k_1,\ell_1),(k_2,\ell_2)$ be two pairs of positive integers with $k_1>k_2,\ 2\le\ell_1<\ell_2$ and $k_i+t(\ell_i-1)+1\le n$ ($i=1,2$), and let $a_1,a_2$ be two positive integers.
		
Now, let $J_2=[u_{2,1},u_{2,a_2}]$ be a segment of $A^t(k_2,\ell_2)$ of cardinality $a_2$. We want to find the admissible values of $a_1$ such that there is a segment $J_1=[u_{1,1},u_{1,a_1}]$ of $A^t(k_1,\ell_1)$ of cardinality $a_1$ with
$$
J_2\cap\BShad_t([u_{1,1},u_{1,a_1}])_{(k_2,\ell_2)}=\emptyset.
$$
		
First of all, $a_1\le|A^t(k_1,\ell_1)|=\binom{k_1+\ell_1-1}{\ell_1-1}$. The inequality must be strict,
otherwise $\BShad_t\big([u_{1,1},u_{1,a_1}]\big)_{(k_2,\ell_2)}=\BShad_t\big(A^t(k_1,\ell_1)\big)_{(k_2,\ell_2)}=A^t(k_2,\ell_2)$, and thus $J_2\cap\BShad_t([u_{1,1},u_{1,a_1}])_{(k_2,\ell_2)}$ $=J_2\ne\emptyset$.
		
Suppose $a_1$ is such that $J_2\cap\BShad_t([u_{1,1},u_{1,a_1}])_{(k_2,\ell_2)}=\emptyset$. Let
$$
v_1:=\min\big\{u\in A^t(k_1,\ell_1):\max J_2=u_{2,1}\notin\BShad_t(u) \big\}.
$$
Clearly, $u_{1,a_1}\ge_{\slex}v_1$. By Theorem \ref{teorcarnum1}, we can calculate the following cardinalities:
\small
\begin{align*}
n_1&:=\big|\big\{ u\in A^t(k_1,\ell_1):u\ge_\slex v_1 \big\}\big|=\big|[x_1x_{1+t}\cdots x_{1+t(\ell_1-2)}x_{k_1+t(\ell_1-1)+1},v_1]\big|,\\
p_1&:=\big|\big\{ v\in A^t(k_1,\ell_1):v>_\slex u_{1,1} \big\}\big|=\big|[x_1x_{1+t}\cdots x_{1+t(\ell_1-2)}x_{k_1+t(\ell_1-1)+1},u_{1,1})\big|.
\end{align*}
\normalsize
Since $[u_{1,1},u_{1,a_1}]\subseteq[x_1x_{1+t}\cdots x_{1+t(\ell_1-2)}x_{k_1+t(\ell_1-1)+1},v_1]$, then
$$
a_1\le n_1.
$$
We can improve such a bound for $a_1$. Indeed, $[u_{1,1},u_{1,a_1}]\subseteq[u_{1,1},v_1]$ and since
\small
\[[u_{1,1},v_1]=[x_1\cdots x_{1+t(\ell_1-2)}x_{k_1+t(\ell_1-1)+1},v_1]\setminus[x_1\cdots x_{1+t(\ell_1-2)}x_{k_1+t(\ell_1-1)+1},u_{1,1}),\]
\normalsize
then we have $|[u_{1,1},v_1]|=n_1-p_1$. Finally,
$$
a_1\le n_1-p_1.
$$
It is clear that if $u_{1,1}=x_1x_{1+t}\cdots x_{1+t(\ell_1-2)}x_{k_1+t(\ell_1-1)+1}=\max A^t(k_1,\ell_1)$, then $p_1=0$, and the bound $a_1\le n_1$ is an optimal bound.
\end{Discus}
	
We want to review Example \ref{carnumes1} \emph{via} Discussion \ref{carnumdisc1}.
\begin{Expl}\rm 
Let $n=13,\ t=2$, $(k_1,\ell_1)=(5,2)$, $(k_2,\ell_2)=(3,4)$, $k_1>k_2,\ 2\le\ell_1<\ell_2, k_i+t(\ell_i-1)+1\le 13$, $i=1,2$, and $a_1\ge 1$, $a_2=10$. 
Let $J_2=[x_2x_4x_6x_{10},x_4x_6x_8x_{10}]$ be a segment of $A^t(k_2,\ell_2)$ of cardinality $a_2=10$. We consider this segment since its monomials are the smallest of $A^t(k_2,\ell_2)$ with respect to $>_{\slex}$. Such a choice gives us the best chance to construct a $t$--spread strongly stable ideal of $S=K[x_1,\dots,x_{13}]$ with extremal Betti numbers $\beta_{k_1,k_1+\ell_1}(I)=a_1$ and $\beta_{k_2,k_2+\ell_2}(I)=a_2=10$.
		
Let us determine the admissible values of $a_1$. Let
\begin{align*}
v_1&=\min\big\{u\in A^t(k_1,\ell_1):\max J_2=u_{2,1}\notin\BShad_t(u) \big\}\\
&=\min\big\{u\in A^2(5,2):x_2x_4x_6x_{10}\notin\BShad_t(u) \big\}\\
&=x_1x_8.
\end{align*}
Hence, $n_1=|[x_1x_8,x_1x_8]|=1$. Let $J_1=[u_{1,1},u_{1,a_1}]$ be a segment of $A^t(k_1,\ell_1)$ of cardinality $a_1$. If $u_{1,1}=x_1x_8$, then $p_1=0$, and $a_1\le n_1-p_1=1$. In such a case $a_1=1$, and we can construct the ideal $I$ as follows:
\begin{align*}
G(I)_{\ell_1}=G(I)_2&=B_2\{ x_1x_8\},\\
G(I)_{\ell_2}=G(I)_4&=B_2\{ [x_2x_4x_6x_{10},x_4x_6x_8x_{10}]\}\setminus\Shad_2^2([I_2]_2).
\end{align*}
The ideal $I$ has two extremal Betti numbers, $\beta_{k_1,k_1+\ell_1}(I)=\beta_{5,7}(I)=1=a_1$ and $\beta_{k_2,k_2+\ell_2}(I)=\beta_{3,7}(I)=10=a_2$.
		
If $u_{1,1}<_\slex x_1x_8$, then $p_1\ge 1$, and $a_1\le n_1-p_1\le 0$ and in such a case the ideal $I$ does not exist. 
One can immediately observe that a $t$--spread strongly stable ideal $I$ of $S$ with extremal Betti numbers $\beta_{k_1,k_1+\ell_1}(I)=a_1=3$ and $\beta_{k_2,k_2+\ell_2}(I)=a_2=10$ does not exist, since $a_1$ can be at most $1$.
\end{Expl}

Discussion \ref{carnumdisc1} will be usefull for the aim of this article.

%
%
%

%

Recall that the floor function of a real number $x$ is defined as follows:
$$
\lfloor x\rfloor=\max\{n\in\ZZ:n\le x\}.
$$
In particular, if $-1\le x<0$ then $\lfloor x\rfloor=-1$ and if $0\le x<1$ then $\lfloor x\rfloor=0$.

We quote the next result from \cite{AFC}. Such a result gives the maximal number of corners allowed for a $t$--spread strongly stable ideal. 
\begin{Thm}\label{maxcornteor} (\cite[Theorem 4.1, Theorem 5.1]{AFC})
Let $n,t, k$ be three positive integers such that $n,t\ge 2$ and $k\ge 3$. Assume
\[n=d+kt, \quad 1\le d\le t.\]
Then, for $k\ge 3$, every $t$--spread strongly stable ideal $I$ of $S=K[x_1,\dots,x_n]$ of initial degree $\ell_1$, $2\le\ell_1\le k+\left\lfloor\frac{d-2}{t}\right\rfloor+1$, and with a corner in degree $\ell_1$ can have at most
$$
\begin{array}{ll}
k+\left\lfloor\frac{d-3}{t}\right\rfloor \, \text{corners}, &\text{if} \ \ell_1=2,\\\\
k+\left\lfloor\frac{d-2}{t}\right\rfloor-(\ell_1-2)\, \text{corners}, &\text{if}\ \ell_1\ge 3.
\end{array}
$$
\end{Thm}
\vspace{0,5cm}

Now we are in position to prove the main result in the paper.

\begin{Thm}\label{carnummainteor}
Let $n,\ell_1,r$ three positive integer, such that $n=d+kt$, $t\ge1$, $1\le d\le t$, $k\ge 3$, $2\le\ell_1\le k+\left\lfloor\frac{d-2}{t}\right\rfloor+1$, and $r$ an integer such that 
$$
\begin{array}{ll}
1\le r\le k+\left\lfloor\frac{d-3}{t}\right\rfloor,&\text{if}\ \ell_1=2,\\
1\le r\le k+\left\lfloor\frac{d-2}{t}\right\rfloor-(\ell_1-2),&\text{if}\ \ell_1\ge 3.
\end{array}
$$
Let $(k_1,\ell_1),\dots,(k_r,\ell_r)$ be $r$ pairs of positive integers such that $n-t-1\ge k_1>k_2>\ldots>k_r\ge 0$, $k_i+t(\ell_i-1)+1\le n$, for all $i=1,\dots,r$,
$$
\begin{array}{ll}
2=\ell_1<\ell_2<\ldots<\ell_r\le k+\left\lfloor\frac{d-3}{t}\right\rfloor+1,&\text{if}\ \ell_1=2,\\\\
3\le\ell_1<\ell_2<\ldots<\ell_r\le k+\left\lfloor\frac{d-2}{t}\right\rfloor+1,&\text{if}\ \ell_1\ge 3,
\end{array}
$$
and let $a_1,\dots,a_r$ be $r$ positive integers. 
The following conditions are equivalent:
\begin{enumerate}
\item[\em(1)] There exists a $t$--spread strongly stable ideal $I$ of $S=K[x_1,\dots,x_n]$ with
$$
\beta_{k_1,k_1+\ell_1}(I)=a_1,\dots,\beta_{k_r,k_r+\ell_r}(I)=a_r
$$
as extremal Betti numbers.
\item[\em(2)]  Setting
\begin{enumerate}
\item[\em(i)] $v_r=x_{k_r+1}x_{k_r+t+1}\cdots x_{k_r+t(\ell_r-1)+1}$ 
and $A_r=[w_r,v_r]$, with $w_r\in A^t(k_r,\ell_r)$ such that $|A_r|=a_r$;
\item[\em(ii)] \small $v_{r-i}=\min\big\{u\in A^t(k_{r-i},\ell_{r-i}):\max A_{r-i+1}\notin\BShad_t(u)_{(k_{r-i+1},\ell_{r-i+1})}\big\}$\normalsize,\\
$A_{r-i}=[w_{r-i},v_{r-i}]$, with $w_{r-i}\in A^t(k_{r-i},\ell_{r-i})$ such that $|A_{r-i}|=a_{r-i}$;
\item[\em(iii)] for $i=1,\dots,r,$ $n_i=\big|\big\{u\in A^t(k_i,\ell_i):u\ge v_i\big\}\big|$, then the integers $a_i$ satisfy the following bound:
$$
a_i\le n_i.
$$
Moreover, if $a_i=|[u_{i,1},u_{i,a_i}]|$, with $u_{i,j}\in A^t(k_i,\ell_i)$ for all $i=1,\dots,r$, and all $j=1,\dots,a_i$, setting $p_i=\big|\big\{v\in A^t(k_i,\ell_i):v>u_{i,1}\big\}\big|$, then $a_i$ satisfy the bound $a_i\le n_i-p_i$, for all $i=1,\dots,r$.
\end{enumerate}
\end{enumerate}
\end{Thm}
\begin{proof}
(1) $\Longrightarrow$ (2). It follows by applying iteratively Discussion \ref{carnumdisc1}.\\
(2) $\Longrightarrow$ (1). We construct a $t$--spread strongly stable ideal $I$ of $S=K[x_1,\dots,x_n]$ generated in degrees $\ell_1,\dots,\ell_r$, with
$$
\beta_{k_1,k_1+\ell_1}(I)=a_1,\dots,\beta_{k_r,k_r+\ell_r}(I)=a_r
$$
as extremal Betti numbers.
		
First of all, we set
$$
G(I)_{\ell_1}:=B_t\{ u_{1,1},\dots,u_{1,a_1}\},
$$
where $u_{1,1}>_\slex\dots>_\slex u_{1,a_1}$ are the first $a_1$ monomials of $A^t(k_1,\ell_1)$ with respect to the squarefree lexicographic order. The existence of such monomials is guaranteed by the hypothesis $a_1\le n_1\le|A^t(k_1,\ell_1)|$. We observe that $[I_{\ell_1}]_t=G(I)_{\ell_1}$ is a $t$--spread strongly stable set, and by Corollary \ref{AC1 cor2} 
\small
$$
\beta_{k_1,k_1+\ell_1}(I)=\Big|\Big\{ u\in G(I)_{\ell_1}:\max(u)=k_1+t(\ell_1-1)+1 \Big\}\Big|=\big|[u_{1,1},u_{1,a_{1}}]\big|=a_1.
$$
\normalsize

Now, we construct $G(I)_{\ell_2}$. Setting $p_1=\big|\big\{v\in A^t(k_1,\ell_1):v>u_{1,1}\big\}\big|$, then $p_1=0$ because $u_{1,1}=x_1x_{1+t}\cdots x_{1+t(\ell-2)}x_{k+t(\ell-1)+1} = \max A^t(k_1,\ell_1)$.
By hypothesis (iii), $a_1\le n_1-p_1=n_1$, i.e.
\begin{align*}
\big|[u_{1,1},u_{1,a_1}]\big|&\le\big|[x_1x_{1+t}\cdots x_{1+t(\ell-2)}x_{k+t(\ell-1)+1},v_1]\big|\\&=\big|[u_{1,1},v_1]\big|.
\end{align*}
It follows that $u_{1,a_1}\ge_{\slex}v_1=\min\big\{v\in A^t(k_1,\ell_1):\max A_2\notin\BShad_t(v)_{(k_2,\ell_2)} \big\}$. Thus, $\max A_2\notin\BShad_t(u_{1,a_1})_{(k_2,\ell_2)}$, and $\min\BShad_t(u_{1,a_1})_{(k_2,\ell_2)}>_\slex \max A_2$. By Remark \ref{carnumoss1},
$$
\min\BShad_t\big(B_t\{[u_{1,1},u_{1,a_1}]\}\big)_{(k_2,\ell_2)}=\min\BShad_t\big(B_t\{ u_{1,a_1}\}\big)_{(k_2,\ell_2)}.
$$
Moreover, since $[I_{\ell_1}]_t=B_t\{[u_{1,1},u_{1,a_1}]\}$, we have
$$
\min\BShad_t\big(B_t\{ [u_{1,1},u_{1,a_1}]\}\big)_{(k_2,\ell_2)}=\min\BShad_t\big([I_{\ell_1}]_t\big)_{(k_2,\ell_2)}>_\slex \max A_2.
$$
Then, $A_2\cap \BShad_t([I_{\ell_1}]_t)_{(k_2,\ell_2)}=\emptyset$ and so we can find at least $a_2$ monomials of $A^t(k_2,\ell_2)$ not belonging to $\BShad_t([I_{\ell_1}]_t)_{(k_2,\ell_2)}$. Let $u_{2,1},\dots,u_{2,a_2}$ be the first $t$--spread monomials of $A^t(k_2,\ell_2)$ not belonging to $ \BShad_t([I_{\ell_1}]_t)_{(k_2,\ell_2)}$, with respect to $>_\slex$ . We set
$$
G(I)_{\ell_2}:=B_t\{ u_{2,1},\dots,u_{2,a_2}\}\setminus \Shad_t^{\ell_2-\ell_1}([I_{\ell_1}]_t).
$$
By our assumptions, from Corollary \ref{AC1 cor2}, we have
\small
$$
\beta_{k_2,k_2+\ell_2}(I)=\Big|\Big\{ u\in G(I)_{\ell_2}:\max(u)=k_2+t(\ell_2-1)+1 \Big\}\Big|=\big|[u_{2,1},u_{2,a_{2}}]\big|=a_2.
$$
\normalsize
One can note that $[I_{\ell_2}]_t$ is a $t$--spread strongly stable set. In fact,
$$
[I_{\ell_2}]_t=B_t\{ u_{2,1},\dots,u_{2,a_2}\}\cup\Shad_t^{\ell_2-\ell_1}([I_{\ell_1}]_t),
$$
and the union of two $t$--spread strongly stable sets is still a $t$--spread strongly stable set.
		
Now, we construct $G(I)_{\ell_3}$.	
Let $a_2=|[u_{2,1},u_{2,a_2}]|$ and let $p_2=|\{p\in A^t(k_2,\ell_2):v>_\slex u_{2,1}\}|$. By hypothesis (iii), $a_2\le n_2-p_2$. Hence,
\begin{align*}
\big|[u_{2,1},u_{2,a_2}]\big|&\le\big|[x_1x_{1+t}\cdots x_{1+t(\ell_2-2)+1}x_{k_2+t(\ell_2-1)+1},v_2]\big|\\&-\big|[x_1x_{1+t}\cdots x_{1+t(\ell_2-2)+1}x_{k_2+t(\ell_2-1)+1},u_{2,1})\big|\\
	&=\big|[u_{2,1},v_2]\big|.
\end{align*}	
It follows that \small $u_{2,a_2}\ge_{\slex} v_2=\min\big\{v\in A^t(k_2,\ell_2):\max A_3\notin\BShad_t(v)_{(k_3,\ell_3)} \big\}$\normalsize, and thus $\max A_3\notin\BShad_t(u_{2,a_2})_{(k_3,\ell_3)}$ and $\min\BShad_t(u_{2,a_2})_{(k_3,\ell_3)}>_\slex \max A_3$. By Remark \ref{carnumoss1},
$$
\min\BShad_t\big(B_t([u_{2,1},u_{2,a_2}])\big)_{(k_3,\ell_3)}=\min\BShad_t(B_t(u_{2,a_2}))_{(k_3,\ell_3)}
$$
and, by construction, we have
$$
\min\BShad_t\big([I_{\ell_2}]_t\big)_{(k_3,\ell_3)}\ge_{\slex}\min\BShad_t\big(B_t([u_{2,1},u_{2,a_2}])\big)_{(k_3,\ell_3)}>_\slex \max A_3.
$$
		
Therefore, $A_3\cap \BShad_t([I_{\ell_2}]_t)_{(k_3,\ell_3)}=\emptyset$ and we can find at least $a_3$ monomials of $A^t(k_3,\ell_3)$ not belonging to $\BShad_t([I_{\ell_2}]_t)_{(k_3,\ell_3)}$. Let $u_{3,1},\dots,u_{3,a_3}$ be the first $t$--spread monomials of $A^t(k_3,\ell_3)$ not belonging to $ \BShad_t([I_{\ell_2}]_t)_{(k_3,\ell_3)}$, with respect to $>_\slex$. 
		
Setting
$$
G(I)_{\ell_3}:=B_t\{ u_{3,1},\dots,u_{3,a_3}\}\setminus \Shad_t^{\ell_3-\ell_2}([I_{\ell_2}]_t),
$$
we have
$$
\beta_{k_3,k_3+\ell_3}(I)=\Big|\Big\{ u\in G(I)_{\ell_3}:\max(u)=k_3+t(\ell_3-1)+1 \Big\}\Big|=\big|[u_{3,1},u_{3,a_{3}}]\big|=a_3.
$$
		
Furthermore, $[I_{\ell_3}]_t$ is $t$--spread strongly stable set. Indeed,
$$
[I_{\ell_3}]_t=B_t\{ u_{3,1},\dots,u_{3,a_3}\}\cup\Shad_t^{\ell_3-\ell_2}([I_{\ell_2}]_t).
$$
		
Clearly, we can iterate this construction, setting
$$
G(I)_{\ell_i}:=B_t\{ u_{i,1},\dots,u_{i,a_i}\}\setminus \Shad_t^{\ell_i-\ell_{i-1}}([I_{\ell_{i-1}}]_t), 
$$
for $i=4, \ldots, r$.
At the end, we will obtain a $t$--spread strongly stable ideal $I$ of $S=K[x_1,\dots,x_n]$ generated in degrees $\ell_1,\dots,\ell_r$ with $\beta_{k_1,k_1+\ell_1}(I)=a_1,\dots,\beta_{k_r,k_r+\ell_r}(I)=a_r$ as extremal Betti numbers.
\end{proof}
\vspace{0,3cm}

We close this Section with a method for computing the minimum of a Borel $t$--shadow. The determination of such a monomial has a key role in the proof of Theorem \ref{carnummainteor}.
%
	
Firstly, we consider a simple case. 
\begin{Lem}\label{lemmacarnum2}
Let $(k_1,\ell_1),(k_2,\ell_2)$ be a pair of positive integers, with $k_1>k_2$, $\ell_1\ge2$, $\ell_2 = \ell_1+1$ and $k_i+t(\ell_i-1)+1\le n$, $i=1,2$. If $u=x_{i_1}x_{i_2}\dots x_{i_{\ell_1}}$ is a monomial of $A^t(k_1,\ell_1)$, then
$$
\min\BShad_t(u)_{(k_2,\ell_1+1)}=x_{i_1}\cdots x_{i_{(\ell_1-2-m)}}\bigg(\prod_{\nu=\ell_1-m-1}^{\ell_1+1}x_{k_2+t(\nu-1)+1}\bigg),
$$
with
\small
$$m=\min\Big\{j\in\ZZ_{\geq -1}:(k_2+t(\ell_2-1)+1-i_{(\ell_1-1)})+\sum_{s=0}^j\wdt\big((\ell_1-2-s)\gap_t\big)(u)\ge t\Big\}.$$
\normalsize
\end{Lem}
\begin{proof}
Firstly, we have that $i_{\ell_1}=k_1+t(\ell_1-1)+1$. Moreover, since $k_1>k_2$, we may consider the monomial
\begin{equation}\label{eq1:min}
v=x_{i_1}x_{i_2}\cdots x_{i_{(\ell_1-1)}}x_{k_2+t(\ell_1-1)+1}x_{k_2+t((\ell_1+1)-1)+1}.
\end{equation}
If $k_2+t(\ell_1-1)+1-i_{(\ell_1-1)}\ge t$, then $v$ is the minimum of $\BShad_t(u)_{(k_2,\ell_1+1)}$. 
		
Otherwise, if $k_2+t(\ell_1-1)+1-i_{(\ell_1-1)}<t$, we proceed as follows.
		
Let us define the following nonnegative integer
\small
$$
\overline{m}=\min\Big\{j\in\ZZ_{\geq 0}:\big(k_2+t\big((\ell_1+1)-1\big)+1-i_{(\ell_1-1)}\big)+\sum_{s=0}^j\wdt\big((\ell_1-2-s)\gap_t\big)(u)\ge t\Big\}.
$$
\normalsize		
		
Clearly, $\overline{m}\in\{0,\dots,\ell_1-3\}$. Hence, the minimum of $\BShad_t(u)_{(k_2,\ell_1+1)}$ is the $t$--spread monomial,
\begin{equation}\label{eq2:min}
 x_{i_1}\cdots x_{i_{(\ell_1-2-\overline{m})}}\bigg(\prod_{\nu=\ell_1-\overline{m}-1}^{\ell_1+1}x_{k_2+t(\nu-1)+1}\bigg).
\end{equation}


More in details, in order to obtain $\min(\BShad_t(u)_{(k_2,\ell_1+1)})$, we decrease the minimum admissible number of variables to the right of $x_{k_2+t(\ell_1-1)+1}$ until we obtain a $t$--spread monomial. 

Finally, setting \small
$$m=\min\Big\{j\in\ZZ_{\geq -1}:(k_2+t(\ell_2-1)+1-i_{(\ell_1-1)})+\sum_{s=0}^j\wdt\big((\ell_1-2-s)\gap_t\big)(u)\ge t\Big\},$$
\normalsize
the assertion follows. Note that if we set $\overline{m}=-1$ in (\ref{eq2:min}), then we obtain the monomial $v$ in (\ref{eq1:min}). This is the case when $k_2+t(\ell_1-1)+1-i_{(\ell_1-1)}\ge t$.
\end{proof}
	
\begin{Rem}\rm
Observe that the integer $m$ in Lemma \ref{lemmacarnum2} can be written in a more suitable way. In fact,
setting 
\[
\widetilde m = \big(k_2+t\big((\ell_1+1)-1\big)+1-i_{(\ell_1-1)}\big)+\sum_{s=0}^j\wdt\big((\ell_1-2-s)\gap_t\big)(u),
\]
one has
\begin{align*}
\widetilde m = &k_2+t\ell_1+1-i_{(\ell_1-1)}+\sum_{s=0}^j\big[i_{(\ell_1-1-s)}-i_{(\ell_1-2-s)}-t\big]\\
=&k_2+t\ell_1+1-i_{(\ell_1-2-j)}-(j+1)t.
\end{align*}
Thus, $$m=\min\Big\{j\in\ZZ_{\geq -1}:k_2+t\ell_1+1-i_{(\ell_1-2-j)}\ge(j+2)t\Big\}.$$
\end{Rem}\medskip

Lemma \ref{lemmacarnum2} assures the correctness of the following algorithm.
\begin{Const}\label{carnumcostr}\rm
Let $n$, $t$ be two positive integers. Let $(k_1,\ell_1),\ (k_2,\ell_2)$ be two pairs of positive integers such that $k_1>k_2,\ 2\le\ell_1<\ell_2$ and $k_i+t(\ell_i-1)+1\le n$, for $i=1,2$.
		
Let $u=x_{i_1}x_{i_2}\cdots x_{i_{\ell_1}}$, $1\le i_1<i_2<\dots <i_{\ell_1}\le n$ be a monomial of $A^t(k_1,\ell_1)$. Hence, $i_{\ell_1}=k_1+t(\ell_1-1)+1$. We want to determine
$$
v=\min\BShad_t(u)_{(k_2,\ell_2)}.
$$
By Lemma \ref{lemmacarnum2},  we can get the monomial
$$
\widetilde{v}=\min\BShad_t(u)_{(k_2,\ell_1+1)}=x_{i_1}\cdots x_{i_{(\ell_1-2-m)}}\bigg(\prod_{\nu=\ell_1-m-1}^{\ell_1+1}x_{k_2+t(\nu-1)+1}\bigg),
$$
with $m=\min\Big\{j\in\ZZ_{\geq -1}:k_2+t\ell_1+1-i_{(\ell_1-2-m)}\ge(m+2)t\Big\}$. Hence, if $\ell_2=\ell_1+1$, then $v=\widetilde{v}$. 
		
Otherwise, 
\begin{align*}
v&=\min\BShad_t(u)_{(k_2,\ell_2)} =\widetilde{v}\cdot x_{k_2+t[(\ell_1+2)-1]+1}\cdots x_{k_2+t(\ell_2-1)+1}\\
&=x_{i_1}\cdots x_{i_{\ell_1-2-m}}\bigg(\prod_{\nu=\ell_1-m-1}^{\ell_1+1}x_{k_2+t(\nu-1)+1}\bigg)x_{k_2+t[(\ell_1+2)-1]+1}\cdots x_{k_2+t(\ell_2-1)+1}.
\end{align*}
Finally,
$$
\min\BShad_t(u)_{(k_2,\ell_2)}=x_{i_1}\cdots x_{i_{(\ell_1-2-m)}}\bigg(\prod_{\nu=\ell_1-m-1}^{\ell_2}x_{k_2+t(\nu-1)+1}\bigg).
$$
\end{Const}\medskip
	
	
\begin{Expl}\rm
Let $S=K[x_1, \ldots, x_{19}]$, $t=4$, $(k_1,\ell_1)=(7,3)$, $(k_2,\ell_2)=(6,4)$. Let $u=x_{i_1}x_{i_2}x_{i_3}=x_6x_{10}x_{16}$ be a monomial of $A^t(k_1,\ell_1)=A^4(7,3)$. We want to determine
		
$$
\min\BShad_t(u)_{(k_2,\ell_2)}=\min\BShad_4(x_6x_{10}x_{16})_{(6,4)}.
$$
We will use Construction \ref{carnumcostr}. One can observe that
$$
k_2+t(\ell_2-1)+1-i_{\ell_1-1}=19-10=9>4=t.
$$
Hence, we have $m=-1$ and so  we do not need to consider the widths of any $j\gap_t$ of $u$. Therefore, 
\begin{align*}
\min\BShad_4(x_6x_{10}x_{16})_{(6,4)}&=x_{i_1}\cdots x_{i_{(\ell_1-2-m)}}\bigg(\prod_{\nu=\ell_1-m-1}^{\ell_2}x_{k_2+t(\nu-1)+1}\bigg)\\
&=x_{i_1}x_{i_2}\Big(\prod_{\nu=3}^{4}x_{6+4(\nu-1)+1}\Big)\\
&=x_6x_{10}x_{15}x_{19}.
\end{align*}
\end{Expl}
	
\section{An example}\label{sec:expl}
In this Section, we provide an example that illustrates Theorem \ref{carnummainteor}.
\begin{Expl}\rm
Let $n=25$, $t=3$, 
\[
\C=\big\{(k_1,\ell_1),(k_2,\ell_2),(k_3,\ell_3),(k_4,\ell_4)\big\}=\big\{(6,2),(5,4),(4,5),(3,7)\big\}\]
and 
\[a=(a_1,a_2,a_3,a_4)=(2,1,3,2).\]
	
Firstly, we \emph{decompose} $n$ as the sum $d+kt$ with $d=1$ and $k=8$, \emph{i.e.}, $n= 25 = 1+ 8\cdot 3$. 
	
On the other hand, $n-t-1=21$ and, since $\ell_1=2$, we have to consider the integer $k+\left\lfloor\frac{d-3}{t}\right\rfloor+1=8$. It follows that the required bounds
$$
n-t-1 \ge k_1>k_2>k_3>k_4\ge1,\quad
2\le\ell_1<\ell_2<\ell_3<\ell_4\le k+\left\lfloor\frac{d-3}{t}\right\rfloor+1
$$ 
are satisfied. According to Theorem \ref{carnummainteor}, a $3$--spread strongly stable ideal $I$ of $S=K[x_1,\dots,x_{25}]$ with corner sequence $\C$ and corner values sequence $a$ does exist if and only if condition (2) in the theorem is satisfied.
Hence, we need to verify the bounds $a_i\le n_i$ for all $i=1,\dots,4$.
	
First of all, let $v_4=x_{k_4+1}\cdots x_{k_4+t(\ell_4-1)+1}=x_4x_7x_{10}x_{13}x_{16}x_{19}x_{22}$ be the smallest monomial of $A^t(k_4,\ell_4)=A^3(3,7)$. Let $w_4 \in A^3(3,7)$ such that $A_4=[w_4,v_4]$ has cardinality $a_4=2$. We have $w_4=x_3x_7x_{10}x_{13}x_{16}x_{19}x_{22}$. On the other hand,
\small
\begin{eqnarray*}
n_4&=&\big|\big\{u\in A^3(3,7):u\ge v_4\big\}\big|=\big|[x_1x_{4}x_7x_9x_{11}x_{14}x_{22},x_3x_7x_{10}x_{13}x_{16}x_{19}x_{22}]\big|\\
&=&|A^3(3,7)|=\binom{3+7-1}{7-1}=\binom{9}{6}=84,
\end{eqnarray*}
\normalsize
and $2=a_4\le n_4=84$.
	
Let
\begin{align*}
v_3&=\min\big\{u\in A^t(k_3,\ell_3):\max A_4\notin\BShad_t(u)_{(k_4,\ell_4)}\big\}\\
&=\min\big\{u\in A^3(4,5):x_3x_7x_{10}x_{13}x_{16}x_{19}x_{22}\notin\BShad_3(u)_{(3,7)}\big\}\\
&=x_3x_6x_{11}x_{14}x_{17}.
\end{align*}

Let $w_3 \in A^3(4,5)$ be the monomial such that $\vert A_3\vert = \vert [w_3,v_3]\vert =a_3=3$. It is $w_3=x_3x_6x_{10}x_{13}x_{17}$. Hence,
$$
A_3=\big\{x_3x_6x_{10}x_{13}x_{17},x_3x_6x_{10}x_{14}x_{17},x_3x_6x_{11}x_{14}x_{17}\big\}.
$$ 

In order to calculate
$$
n_3=\big|\big\{u\in A^3(4,5):u\ge v_3\big\}\big| = \big|[x_1x_4x_7x_{10}x_{17},x_3x_6x_{11}x_{14}x_{17}]\big|,
$$ 
we consider the highlighted binomial coefficients of the following binomial decompositions:
\begin{align*}
\tbinom{k_3+\ell_3-1}{\ell_3-1}=\tbinom{8}{4}=\boxed{\bf{\tbinom{7}{3}+\tbinom{6}{3}}} + &\tbinom{5}{3}+ \tbinom{4}{3}+\tbinom{3}{3}\\
& \tbinom{5}{3}=
\begin{aligned}[t]  \tbinom{4}{2}&+\tbinom{3}{2}+\tbinom{2}{2} \\
\tbinom{4}{2}&=\boxed{\bf{\tbinom{3}{1}+\tbinom{2}{1}}}+\tbinom{1}{1}.
\end{aligned}
\end{align*}
Hence, $n_3=\tbinom{7}{3}+\tbinom{6}{3}+\tbinom{3}{1}+\tbinom{2}{1}+\tbinom{0}{0}=61$ and $3=a_3\le n_3=61$.
	
Let
\begin{align*}
v_2&=\min\big\{u\in A^t(k_2,\ell_2):\max A_3\notin\BShad_t(u)_{(k_3,\ell_3)}\big\}\\
&=\min\big\{u\in A^3(5,4):x_3x_6x_{10}x_{13}x_{17}\notin\BShad_3(u)_{(4,5)}\big\}\\
&=x_3x_6x_9x_{15}.
\end{align*}
Since $a_2=1$, we have $A_2=[w_2,v_2]=\{v_2\}$, \emph{i.e.}, $w_2=v_2$. Now, we evaluate $n_2=\big|\big\{u\in A^3(5,4):u\ge v_2\big\}\big| =\big|[x_1x_4x_7x_{15},x_3x_6x_9x_{15}]\big|$. We consider the following binomial decompositions:
\begin{equation}\label{decompesmainteor}
\begin{aligned}
\tbinom{k_2+\ell_2-1}{\ell_2-1}=\tbinom{8}{3}=\boxed{\bf{\tbinom{7}{2}+\tbinom{6}{2}}} + &\tbinom{5}{2}+ \tbinom{4}{2}+\tbinom{3}{2}+\tbinom{2}{2}\\
& \tbinom{5}{2}=
\begin{aligned}[t]  \tbinom{4}{1}&+\tbinom{3}{1}+\tbinom{2}{1}+\tbinom{1}{1} \\
\tbinom{4}{1}&=\tbinom{3}{0}+\tbinom{2}{0}+\tbinom{1}{0}+\tbinom{0}{0}.
\end{aligned}
\end{aligned}
\end{equation}
Hence, $n_2=\tbinom{7}{2}+\tbinom{6}{2}+\tbinom{0}{0}=37$ and $1=a_2\le n_2=37$.
	
Finally, $v_1=\min\big\{u\in A^t(k_1,\ell_1):\max A_2\notin\BShad_t(u)_{(k_2,\ell_2)}\big\}=x_2x_{10}$. Since $a_1=2$, we have $A_2=[x_1x_{10},x_2x_{10}]=\{x_1x_{10},x_2x_{10}\}$, and $a_1=2=n_1$.\medskip

Now, we construct a $3$--spread strongly stable ideal $I$ with corner sequence $\Corn(I) = \C$ and $a(I) = a$ by refining the bounds for the $a_i$'s (see Theorem \ref{carnummainteor}, (iii)).
\begin{enumerate}
\item[-] $a_1=2$, since $p_1=0$, $a_1=n_1-p_1=n_1=2$, we set 
$$G(I)_{\ell_1}=G(I)_2=B_3\{x_1x_{10},x_2x_{10}\}.$$
Note that $x_1x_{10}, x_{2}x_{10}$ are the greatest monomials of $A^t(k_1,\ell_1)=A^3(6,2)$, with respect to $>_{\slex}$.
\item[-] Let us consider the corner $(k_2,\ell_2)=(5,4)$. Since $\min\BShad_3([I_2]_3)_{(5,4)}$ $=$ $x_2x_9x_{12}x_{15}$, then  $\max(A^t(k_2,\ell_2)\setminus\BShad_t([I_{\ell_1}]_2)_{(5,4)})$ is the largest monomial of $A^t(k_2,\ell_2)$ following $x_2x_9x_{12}x_{15}$ respect to $>_{\slex}$, \emph{i.e.}, $u_{2,1}$ $=$ $x_3x_6x_9x_{15}$. In order to calculate $p_2=|\{v\in A^3(5,4):v>u_{2,1}\}|$, we need to consider the same binomial decompositions of (\ref{decompesmainteor}). Thus, $p_2=\binom{7}{2}+\binom{6}{2}=36$, $a_2=1\le n_2-p_2=37-36=1$. Hence, we set
$$
G(I)_{\ell_2}=G(I)_4=B_3\{ x_3x_6x_9x_{15}\}\setminus\Shad_3^2([I_2]_3).
$$
\item[-] Let us consider the corner $(k_3,\ell_3)=(4,5)$. Since $\min\BShad_3([I_4]_3)_{(4,5)}$ $=$ $x_3x_6x_9x_{14}x_{17}$, then $\max(A^t(k_3,\ell_3)\setminus\BShad_3([I_4]_3)_{(4,5)})$ is the largest monomial of $A^t(k_3,\ell_3)$ that follows $x_3x_6x_9x_{14}x_{17}$ respect to  $>_{\slex}$. It is $u_{3,1}=x_3x_6x_{10}x_{13}x_{17}$. In order to evaluate $p_3=|\{v\in A^3(4,5):v>u_{3,1}\}|$, we consider the following suitable binomial decompositions:
\begin{align*}
\tbinom{k_3+\ell_3-1}{\ell_3-1}=\tbinom{8}{4}=\boxed{\bf{\tbinom{7}{3}+\tbinom{6}{3}}} + &\tbinom{5}{3}+ \tbinom{4}{3}+\tbinom{3}{3}\\
& \tbinom{5}{3}=
\begin{aligned}[t]  \tbinom{4}{2}&+\tbinom{3}{2}+\tbinom{2}{2} \\
\tbinom{4}{2}&=\boxed{\bf{\tbinom{3}{1}}}+\tbinom{2}{1}+\tbinom{1}{1}.
\end{aligned}
\end{align*}
Therefore, $p_3=\tbinom{7}{3}+\tbinom{6}{3}+\tbinom{3}{1}=58$ and $3=a_3\le n_3-p_3=61-58=3$, and we can construct $G(I)_{\ell_3}=G(I)_5$. We set
$$
G(I)_5=B_3\{ x_3x_6x_{10}x_{13}x_{17},x_3x_6x_{10}x_{14}x_{17},x_3x_6x_{11}x_{14}x_{17}\}\setminus\Shad_3^1([I_4]_3).
$$
\item[-] Finally, we consider the corner $(k_4,\ell_4)=(3,7)$. The smallest monomial of $\BShad_3([I_5]_3)_{(3,7)}$ is $x_3x_6x_{10}x_{13}x_{16}x_{19}x_{22}$, then the largest monomial of $A^t(k_4,\ell_4)\setminus\BShad_3([I_5]_3)_{(3,7)}$ is $u_{4,1}=x_3x_7x_{10}x_{13}x_{16}x_{19}x_{22}$. In order to compute $p_4=|\{v\in A^3(3,7):v>u_{4,1}\}|$, we consider the following binomial decompositions:
\begin{align*}
\tbinom{9}{6}=\boxed{\bf{\tbinom{8}{5}+\tbinom{7}{5}}}+&\tbinom{6}{5}+\tbinom{5}{5}\\
& \tbinom{6}{5}= \begin{aligned}[t]
\boxed{\bf{\tbinom{5}{4}}}+\tbinom{4}{4}.
\end{aligned}
\end{align*}
Hence $p_4=\binom{8}{5}+\binom{7}{5}+\binom{5}{4}=82$, and $2=a_4\le n_4-p_4=84-82=2$. Thus, we can construct $G(I)_{\ell_4}=G(I)_7$. We set
$$
G(I)_7=B_3\{ x_3x_7x_{10}x_{13}x_{16}x_{19}x_{22},x_4x_7x_{10}x_{13}x_{16}x_{19}x_{22}\}\setminus\Shad_3^2([I_5]_3).
$$
\end{enumerate}
Finally, we have constructed a $3$--spread strongly stable ideal
\begin{align*}
I=&B_3\big(x_1x_{10},\bm{x_2x_{10}},\bm{x_3x_6x_9x_{15}},x_3x_6x_{10}x_{13}x_{17},x_3x_6x_{10}x_{14}x_{17},\bm{x_3x_6x_{11}x_{14}x_{17}},\\&x_3x_7x_{10}x_{13}x_{16}x_{19}x_{22},\bm{x_4x_7x_{10}x_{13}x_{16}x_{19}x_{22}}\big)\\
=&\big(x_1x_4,x_1x_5,x_1x_6,x_1x_7,x_1x_8,x_1x_9,x_1x_{10},x_2x_5,x_2x_6,x_2x_7,x_2x_8,x_2x_9,\bm{x_2x_{10}},\\
&x_3x_6x_9x_{12},x_3x_6x_9x_{13},x_3x_6x_9x_{14},\bm{x_3x_6x_9x_{15}},\\
&x_3x_6x_{10}x_{13}x_{16},x_3x_6x_{10}x_{13}x_{17},x_3x_6x_{10}x_{14}x_{17},\bm{x_3x_6x_{11}x_{14}x_{17}},\\
&x_3x_7x_{10}x_{13}x_{16}x_{19}x_{22},\bm{x_4x_7x_{10}x_{13}x_{16}x_{19}x_{22}}\big).
\end{align*}
with corner sequence $\C$ and corner values sequence $a$. The highligthed monomial are the $3$--spread Borel generators of $I$. The Betti table of $I$ is
\begin{center}
\begin{tabular}{ccccccccc}
&&0&1&2&3&4&5&6\\ 
Tot&:&23&77&117&100&51&15&2\\ \hline
2&:&13&42&70&70&42&14&2\\
3&:&-&-&-&-&-&-&-\\
4&:&4&14&20&15&6&1&-\\
5&:&4&15&21&13&3&-&-\\
6&:&-&-&-&-&-&-&-\\
7&:&2&6&6&2&-&-&-
\end{tabular}
\end{center}
\end{Expl}

\begin{Rem}\em
All the algorithmic and constructive methods analyzed in this paper have been implemented into a \emph{Macaulay2} packages ``TSpreadIdeals.m2'', by the authors themselves, and tested with \emph{Macaulay2} $1.17.2.1$. To the best of our knowledge, specific packages for manage $t$-spread ideals are not yet available in the most well--known symbolic algebra systems (\emph{e.g.}, \emph{Macaulay2}, CoCoA or Singular). Hence, we believe that such a package could be useful since the topics covered have been introduced recently and represent a new line of research (\cite{AC2, AEL,CAC, EHQ, RD}).

All the examples in this paper have been checked or constructed by this package.
\end{Rem}

\end{document}